\documentclass[a4paper]{article}
\usepackage[utf8]{inputenc}
\usepackage{amsmath,amsthm,amssymb,amsfonts}
\usepackage[english, russian]{babel}
\usepackage{chngcntr}
\usepackage{cite}
\usepackage{mathtools}
\usepackage{extpfeil}

\theoremstyle{plain}
\newtheorem{theorem}{Theorem}[section]
\newtheorem{lemma}[theorem]{Lemma}

\newtheorem{definition}[theorem]{Definition}
\newtheorem{hypoth}[theorem]{Hypothesis}

\newtheorem*{remark}{Remark}

\counterwithout{theorem}{section}
\addto\captionsrussian{}
\addto\captionsrussian{}

\title {
    Concentration of the hypergraph's weak independence number
}
\author {\Large
    S.Vakhrushev\footnote{Uppsala University, Uppsala, Sweden; Moscow institute of Physics and Technology, Dolgoprudny}
}

\date{}

\usepackage[]{hyperref}

\usepackage{subfigure} 

\newcounter{image}
\newenvironment{image}[1][]{\refstepcounter{image} 
	\small Fig.~\theimage. #1 \rmfamily}{}

\usepackage{pgfplots}

{\par\noindent{\it Proof for the first case:}} 
{\hfill$\scriptstyle\blacksquare$} 

{\par\noindent{\it Proof for the second case:}} 
{\hfill$\scriptstyle\blacksquare$} 

\begin{document}

\maketitle

\renewcommand{\abstractname}{Abstract}
\begin{abstract}
    In this note we generalize the results of the recent work by Tom Bohman and Jakob Hofstad on the independence number in G(n, p) to the case of the random k-uniform hypergraph. The weak independence number is concentrated on two values for $p > n^{-\frac{(k-1)k}{k+1}+\varepsilon}$. This bound is roughly optimal, as such concentration is generally not observed when $p=o\left(n^{-\frac{(k-1)k}{k+1}}\log^{\frac{2}{k+1}} n\right)$.

Bibliography: 17 titles.
\end{abstract}

\vspace{\baselineskip}


\section{Introduction}
In this paper, we study the limit distribution of the independence number in random hypergraphs.
\subsection{Necessary Definitions from Hypergraph Theory}\label{s1.1}

A \textit{hypergraph} is defined as a pair $H=(V, E)$, where $V=V(H)$ --- a finite set whose elements are called \textit{vertices of the hypergraph}, and $E=E(H)$ --- an arbitrary collection of subsets of $V$, called \textit{edges of the hypergraph} $H$. A hypergraph is called \textit{$k$-uniform} if every edge $A \in E$ has cardinality $k$. For example, a $2$-uniform hypergraph is an ordinary graph without loops and multiple edges.

Let $j$ be a positive integer. A subset of vertices $S \subset V$ in a hypergraph $H = (V, E)$ is called $j$-independent if for any edge $A \in E$, it holds that $|A \cap S| \leq j$, i.e., each edge has at most $j$ common vertices with $S$. The \textit{$j$-independence number} $\alpha_j(H)$ of a hypergraph $H$ is defined as the maximum size of $j$-independent sets in $H$.

Note that if the hypergraph is $k$-uniform, it suffices to consider only $1 \leq j \leq k-1$, since for larger $j$ all subsets of vertices will be independent. When $j = 1$ (an edge has at most one vertex in the independent set), $1$-independent sets in a uniform hypergraph are usually called strong independent sets, and in the case $j = k - 1$ (an edge is not fully contained in the independent set), $(k - 1)$-independent sets are called weak independent sets. For graphs, when $k=2$, these two concepts coincide. In this paper, unless otherwise stated, we always consider independence as weak independence, as is commonly done in hypergraph theory, and denote $\alpha(H)=\alpha_{k-1}(H)$.

In the present paper, we deal with \textit{random hypergraphs} in the classical binomial model $H(n, k,p)$, where $n > k \geq 2$, $n, k \in \mathbb{N}$, and $p\in (0, 1)$. These are obtained by independently including each $k$-subset of a fixed set of $n$ vertices $V_n$ with probability $p$ (Bernoulli scheme on $k$-subsets of $V_n$). More formally, a random hypergraph is a random element taking values in the set of all $k$-uniform hypergraphs on a fixed set $V_n$ of vertices and having the following distribution: for any $H'=(V_n, E')$,
$$\mathbb{P}(H(n, k, p)=H') = p^{|E'|}(1-p)^{{n \choose k} - |E'|}.$$
When $k=2$, the model $H(n, 2, p)$ is well-known as the binomial random graph model $G(n, p)$, to which a vast number of works have been dedicated (see, for example, the classical books~\cite{bollobas, frieze, janson}). In this article, we assume that $k \geq 2$ is fixed, $n$ tends to infinity, and $p = p(n, k)$ depends on both parameters.

We will say that a random variable $X$, defined on a random hypergraph, is \textit{concentrated on $\ell=\ell(n)$ values} if there exists a sequence $t=t(n)$ such that $X \in \{t, t+1, \ldots, t+\ell-1\}$ a.a.s. (here and throughout: asymptotically almost surely).

\vspace{\baselineskip}

For asymptotic analysis, in addition to the standard notations, we will use the following:

    (1) we write $f(n)=\omega(g(n))$ if $g(n)=o(f(n))$ as $n \xrightarrow[]{} \infty$;

    (2) we write $f(n) = \Omega(g(n))$ if $g(n) = O(f(n))$ as $n \xrightarrow[]{} \infty$;

    (3) we write $f(n) \sim g(n)$ if $\lim\limits_{n\xrightarrow[]{} \infty} f(n)/g(n)=1$.
\noindent{}
Also, throughout, by $\log n$ we mean the natural logarithm.
\subsection{Background}
We begin the review of results on the concentration of the independence number with the random graph $G(n, p)$. The first results on this topic date back to the 1970s, when B. Bollobás and P. Erdős~\cite{bol-erd}, as well as D. Matula~\cite{matula}, independently proved the two-point concentration for constant $p \in (0, 1)$. Then, in the late 1980s, A. Frieze~\cite{frieze_asympt}, for the case of decreasing $p=p(n)$, found the asymptotic behavior of the independence number with rather good precision. He showed that for $\omega(1/n) < p < o(1)$, with high probability,
\begin{equation}
\label{frieze_asympt}
    \alpha(G_{n, p}) = \frac{2}{p} (\log (np) - \log \log (np) + \log(e/2) \pm o(1)).
\end{equation}
However, the order of magnitude of the spread of values is quite large. Recently, T. Bohman and J. Hofstad, in their breakthrough work~\cite{main_paper}, proved two-point concentration of the independence number for $p > n^{\frac{-2}{3}+\varepsilon}$. They also showed that their bound is roughly best possible: at least for some $p \leq (\log (n)/n)^{\frac{2}{3}}$, the independence number is concentrated on at least three values.

Note that in the so-called "sparse" \hspace{1pt} case --- for a fixed product $np=c>0$ --- M. Bayati, D. Gamarnik, and P. Tetali~\cite{boyati} proved the existence of some $\gamma(c) \in (0, 1)$ such that

$$\frac{\alpha(G(n, c/n))}{n} \xrightarrow[]{\sf P} \gamma(c) \text{ when } n \xrightarrow[]{} \infty.$$ 
And T. Bohman and J. Hofstad, in the very same work~\cite{main_paper}, proved that the independence number in this regime cannot be concentrated on fewer than $O(\sqrt{n})$ values.

Finally, quite recently, T. Bohman and J. Hofstad posted on the arXiv a preprint~\cite{recent_paper} with a result on the concentration of the independence number in the uniform random graph model $G(n, m)$: a random element taking values uniformly on the set of all graphs with $n$ vertices and $m$ edges. It is well known that many properties in the uniform and binomial models are inherited under the natural substitution $p=m/{n \choose 2}$. It turns out that the property of two-point concentration of the independence number differs in these models: already for $m > n^{5/4+\varepsilon}$, it holds that $\alpha(G(n, m))$ is concentrated on two values.

\vspace{\baselineskip}
The study of the behavior of the independence number for hypergraphs has mainly been associated with the study of chromatic numbers and has been studied since the 1980s. In the works of J. Schmidt, E. Shamir and co-authors~\cite{shmidt1, shmidt2, shamir}, some estimates of the $j$-independence numbers of $H(n, k, p)$ were established to obtain asymptotics for the $j$-chromatic numbers of random graphs. The most comprehensive study of this question was conducted by M. Krivelevich and B. Sudakov~\cite{sudakov}, who established more precise estimates of the $j$-independence numbers, which allowed them to determine their asymptotic growth:

\begin{theorem}[Consequence of {~\cite[Consequence 1]{sudakov}}]
    Fix $k$ and $1 \leq j \leq k-1$ and let $p=\omega\left(n^{1-k}\right)$ and simultaneously $p=o\left(n^{j+1-k}\right)$. Then a.a.s.
    \begin{equation}
    \label{asymptotic_base}
    \alpha_j(H(n, k, p)) \sim n^{1-\frac{k-1}{j}} \left((j+1) (j-1)! (k-1)(k-1-j)! \frac{\log \left(np^{1/(k-1)}\right)}{p}\right)^{\frac{1}{j}},
    \end{equation}
    and, in particular,
    \begin{equation}
    \label{asymptotic}
    \alpha(H(n, k, p))=\alpha_{k-1}(H(n, k, p))=\left(\frac{k! \log \left(np^{1/(k-1)}\right)}{p}\right)^{\frac{1}{k-1}} (1+o(1)).
    \end{equation}
\end{theorem}
Note, however, that the spread of values in the original estimates~\cite{sudakov} is quite large compared to A. Frieze's formula (\ref{frieze_asympt}) for random graphs. In this regard, it is not surprising that the question of concentration of independence numbers had not been raised in the literature until recently. In the works of A. Semenov and D. Shabanov~\cite{semenov1, semenov2}, the asymptotic behavior of $\alpha_j(H(n, k, p))$ was investigated in the sparse case, where $p = c/ {{n-1} \choose {k-1}}$ and the parameter $c > 0$ depends only on $k$ and $j$. They found the limit of convergence in probability of $\alpha_j(H(n, k, p))/n$ for $c \leq 1 / (k-1)$, and also proved the existence of such a limit for $j = k-1$ and arbitrary $c$.

Recently, I. Denisov and D. Shabanov proved~\cite{denisov} the concentration of the weak independence number $\alpha(H(n, k, p))$ on two values for constant $p \in (0, 1)$. They also proved a result on the concentration of the $j$-independence number of $H(n, k, p)$ on two values in the non-trivial regime $\Omega(n^{j+1-k}) \leq p \leq o\left(n^{j+1-k} \log n\right)$, and proved concentration on a logarithmic number of values for slightly smaller $p$. Their proof uses the second moment method and follows the proof scheme of B. Bollobás and P. Erdős~\cite{bol-erd} for random graphs $G(n, p)$ with constant $p$.

\section{New results}
\subsection{Statement of the main theorems}
The main result of this work is the following:
\begin{theorem}
\label{th:main_res}
    Let $p \geq n^{-\frac{(k-1)k}{k+1} + \varepsilon}$ for any fixed $\varepsilon > 0$. Then $\alpha(H(n, k, p))$ is concentrated on two values.
\end{theorem}

Thus, we generalize the result of T. Bohman and J. Hofstad~\cite{main_paper} for the weak independence number of a random hypergraph. Two-point concentration is observed for a wide range of values of $p$, which improves all previously known results. We also show that the lower bound on $p$ is almost optimal, using the technology of A. Sah and M. Sawhney:

\begin{theorem}
\label{th:antic}
    Let $\omega\left(n^{-(k-1)}\right) < p < \left(\frac{(k! \log (n))^{1/(k-1)}}{6(k+1)^{1/(k-1)}n^{k/2}}\right)^{2(k-1)/(k+1)}$ and set $$\ell = \ell(n, k, p) := \frac{n^{-k/2}p^{-1/2-1/(k-1)} \left(k! \log (np^{1/(k-1)})\right)^{1/(k-1)}}{3}.$$ Then there exists $q = q(n)$ such that $p \leq q \leq 2p$ such that $\alpha(H(n, k, q))$ is not concentrated on $\ell$ values.
\end{theorem}
However, note that this anti-concentration result is somewhat weak since it does not imply that there are no such functions $p(n, k) < n^{-\frac{(k-1)k}{k+1}}$ for which the 2-point concentration holds. Furthermore, it should be noted that in a recent work~\cite{recent_paper}, the absence of such functions for random graphs (i.e., when $k=2$) was already shown by analyzing the concentration of the independence number in $G(n, m)$.

Incidentally, as a corollary of the proof of Theorem~\ref{th:main_res}, a rather accurate asymptotic was obtained, similar to A. Frieze's formula (\ref{frieze_asympt}). To define it we, as in~\cite{denisov}, introduce a function

\begin{equation}
\nonumber
    f(x) = f_k(x) := x(x-1)(x-2)\ldots (x-k+2),
\end{equation}
which is continuous and strictly increasing for $x > k-2$, and therefore one can define an increasing inverse function $f^{-1}(y)$. For further asymptotic analysis, the following obvious observation will be important to us:

\begin{equation}
\nonumber
    f^{-1}(y) \sim \sqrt[k-1]{y} \quad \text{ as } \quad y\xrightarrow[]{} \infty.
\end{equation}
Also throughout this work, for the sake of brevity, we will use the notation
$$d=d(n, k):=np^{\frac{1}{k-1}}.$$
\begin{theorem}
\label{th:asympt}
Let $n^{-\frac{(k-1)k}{k+1}+\varepsilon} \leq p = o(1)$. Then whp
    \begin{equation}
    \alpha(H(n, k, p)) = f^{-1} \left(\frac{k![\log d - \frac{1}{k-1} \log \log d + \log \frac{e}{(k!)^{1/(k-1)}} + o(1)]}{p} \right).
\end{equation}
\end{theorem}
As one can see, the spread of values in this formula has order $\sqrt[k-1]{1/p}$, which significantly refines the result of M. Krivelevich and B. Sudakov (\ref{asymptotic}).

\subsection{Proof technology}
Our work fully follows the scheme of T. Bohman and J. Hofstad's work~\cite{main_paper}. The main tool for proving two-point concentration is the application of the second moment method.

It is natural to consider the number of independence sets of size $s$ in $H(n,k, p)$. Let us denote this value by $X_s$. Clearly,
$$\mathbb{E} [X_s] = {n \choose s} (1-p)^{s\choose k}.$$
If $k$ is large enough, then the expectation of $X_S$ tends to $0$ and with high probability there are no independent sets of size $s$ in $H(n, k, p)$. This gives a simple upper bound for $\alpha(H(n, k, p))$. It turns out that this upper bound is sufficient for proving two-point concentration for $p > n^{-\frac{(k-1)^2}{k} + \varepsilon}$. However, for smaller $p$, this upper bound becomes weak, since the independence number is actually already concentrated on values smaller than the largest $s$ for which $\mathbb{E} [X_s] = \omega(1)$.

\vspace{\baselineskip}

For the range $n^{-\frac{(k-1)k}{k+1}} < p < n^{-\frac{(k-1)^2}{k}}$, T. Bohman and J. Hofstad in~\cite{main_paper} introduced a more general object --- \textit{augmented independent sets} --- which became the main tool for addressing the gap between the expectation threshold and the true independence number in $G(n, p)$, allowing for an improvement of the upper bound. We present a generalization of their definition for hypergraphs:

\begin{definition}
    An augmented independent set of order $s$ in a $k$-uniform hypergraph $H$ is a set of vertices $S$ in $H$ such that for some integer $r \geq 0$ it satisfies the following conditions:
    \begin{itemize}
        \item $|S| = s+r$;
        \item induced hypergraph $H[S]$  has exactly $r$ edges, which are all mutually disjoint;
    \item for any vertex $v \in V(H) \backslash S$, there exist at least two edges containing $v$ within $S \cup \{v\}$.
    \end{itemize}
\end{definition}
The primary motivation for this definition is that each such structure contains exactly $k^r$ independent sets of size $s$, and therefore, considering the number of augmented independent sets of order $s$ instead of ordinary independent sets potentially significantly reduces the variance. Moreover, the third condition is natural in the sense of maximality by inclusion (which allows us to consider more independent sets of size $s$). Indeed, suppose that from some vertex $v \in V(H) \backslash S$, there is at most one edge $e \subset S \cup \{v\}$ to $S$. If such a potential edge does not intersect the edges of a matching within $S$, then we can increase the size of $S$ by adding vertex $v$. If such an edge intersects the existing edges of the matching at any vertex $u$, then the set $S \backslash \{u\} \cup \{v\}$ contains an independent set of size $s+1$.

Thus, the given object is an ``extension'' of the notion of \textit{maximal independent sets}, allowing us to account for a large number of independent sets of a fixed size. A key fact, analogous to ~\cite[Lemma 4]{main_paper}, on which the proof of the refined upper bound relies, is

\begin{lemma}
\label{lm:idea}
    For any $k$-uniform hypergraph $H$, the independence number $\alpha(H)$ is equal to $s$ if and only if $H$ contains an augmented independent set of order $s$, but does not contain augmented independent sets of larger order.
\end{lemma}

\begin{proof}
    Let us denote by $\hat{\alpha}(H)$ the largest order of an augmented independent set. Since any augmented independent set of order $s$ contains an independent set of size $s$, then
    $$\hat{\alpha}(H) \leq \alpha(H).$$
    Now, suppose that $H$ contains a maximum independent set $S$ of order $s$. Let $T$ be a maximal set of vertices containing $S$ such that all edges contained in $T$ are mutually disjoint. We show that $T$ is an augmented independent set of order $s$:
    
    Obviously, from any vertex outside $T$, there is at least one edge into $T$ that intersects the edges of the matching, due to the maximality of $T$. Suppose that for some vertex $v\in V(H) \backslash T$, there is exactly one such edge. Let the intersection of this edge with the edges of the matching occur at vertices $\{u_1, \ldots, u_t\}$, where $1 \leq t \leq k-1$. Then $T \backslash \{u_1, \ldots, u_t\} \cup \{v\}$ contains an independent set of larger size than $S$, which contradicts the maximality of $S$. Thus, from each vertex outside $T$, there are at least two edges into $T$, which proves that

    $$\alpha(H) \leq \hat{\alpha}(H).$$
    The proof of Lemma~\ref{lm:idea} is finsihed.
\end{proof}

We will prove the concentration of the independence number by applying the second moment method to the random variable $Z_{s, r}$, which counts the number of augmented independent sets of order $s$ with $r$ edges of the matching inside. The parameters $s$ and $r$ will be chosen optimally. The bottleneck, where the lower bound on $p$ in Theorem~\ref{th:main_res} is used, is the consideration of pairs of augmented sets with an intersection of approximately $p^{1/k}$ vertices.

\vspace{\baselineskip}

To prove Theorem~\ref{th:antic} regarding the absence of concentration for values of $p$ smaller than in Theorem~\ref{th:main_res}, the argument of A. Sah and M. Sawhney described in~\cite[Section 4]{main_paper} is used. It shows that two-point concentration cannot be generalized for $p=n^{\gamma}$, where $\gamma \leq \frac{(k-1)k}{k+1}$. This approach is very similar to the one in~\cite{anti_approach} (see the last paragraph of the last section) and relies on the normal approximation of the distribution of the number of edges in $H(n, k, p)$ combined with the observation that the property of containing an independent set of a given size is decreasing. This implies an intersection of the concentration intervals of the independence numbers of random $H(n, k, p)$ and $H(n, k, p+\delta)$ for an appropriately chosen step $\delta$.

\section{Structure of the Proof of Theorem~\ref{th:main_res}}

We will only consider the case $p=o(1)$ since for constant $p$ the result of Theorem~\ref{th:main_res} is known from~\cite{denisov}. We define the ``expectation threshold'' $s_x$ as the largest integer $s$ such that

$$\mathbb{E} [X_s] = {n \choose s} (1-p)^{s\choose k} > n^{2\varepsilon}.$$
Note that the exact value of the exponent of the ``cut-off function'' on the right-hand side is not critical: we only need it to be small (specifically, we need $\mathbb{E}[X_{s_x +2}] = o(1)$) but greater than the exponent in the polynomial function we use to define the variable $s_z$ in Section~\ref{sc:tech}.

In this work, we will deal with values of $s$ satisfying the following asymptotic relation:
\begin{equation}
\label{main_fm}
    s =s(n, p) =  f^{-1} \left(\frac{k![\log d - \frac{1}{k-1} \log \log d + \log \frac{e}{(k!)^{1/(k-1)}} + o(1)]}{p} \right).
\end{equation}
Further technical calculations in this paper are mainly based on the asymptotic properties of functions of $s$ having such an order of growth (see Section~\ref{sc:tech}). It is easy to see that $s_x$ can be represented in this form (the influence of $2\varepsilon$ in the definition of $s_x$ is absorbed by the $o(1)$ term in this expression).

\vspace{\baselineskip}

The main technical result in the proof of concentration is the following:
\begin{theorem}
\label{th:main_tech}
    If $n^{-\frac{(k-1)k}{k+1} + \varepsilon} < p < \log^{-(k-1)} n$ for some $\varepsilon > 0$, then there exists an integer $s_z = s_z(n, k)$ such that $\alpha(H(n, k, p)) \in \{s_z, s_z+1\}$ a.a.s.. Furthermore, we have
    $$s_x - s_z \begin{cases}
        = \text{ } 0,\text{ if } p=\omega\left(\log (n)n^{-\frac{(k-1)^2}{k}}\right) \\
        = \xi_n, \text{ if } p=C\log(n)n^{-\frac{(k-1)^2}{k}} \\
        \sim \frac{s_x^k}{e^kn^{k-1}} \sim \frac{(k! \log d)^{\frac{k}{k-1}}}{e^kp^{\frac{k}{k-1}}n^{k-1}},\text{ if } p=o\left(\log (n)n^{-\frac{(k-1)^2}{k}}\right)
    \end{cases}$$
    where $\xi_n \in \left(\frac{((k-1)!)^{\frac{k}{k-1}}}{C^{\frac{k}{k-1}}e^k} - \frac{k}{k-1} - \frac{1}{2}, \frac{((k-1)!)^{\frac{k}{k-1}}}{C^{\frac{k}{k-1}}e^k} + \frac{3}{2}  \right)$.
\end{theorem}

Let us make a few remarks regarding the results of this theorem and its corollaries:
\begin{itemize}
    \item This result is a direct generalisation of the main technical result by T.Bohman and J.Hofstad~\cite{main_paper}[Theorem 3].
    
    \item We emphasize that in the case $p=C\log(n)n^{-\frac{(k-1)^2}{k}}$, the sequence $\xi_n$ does not converge to a specific value. Moreover, it varies within a specified interval. In fact, $\xi_n$ can be defined more precisely by the following expression

$$\left\lceil \frac{((k-1)!)^{\frac{k}{k-1}}}{C^{\frac{k}{k-1}}e^k} - \frac{k \log \mathbb{E} [X_{s_x}]}{(k-1)\log n} + \frac{k}{k-1}\varepsilon \pm o(1)\right\rceil.$$
Since $\mathbb{E} [X_{s_x}]$ can take a wide range of values as $n$ varies, we have persistent variation of $\xi_n$ over
a small list of values.

    \item As can be seen, for $p > n^{-\frac{(k-1)^2}{k}+\varepsilon}$ a simple upper bound $s_x$ given by the first moment method is sufficient. We show in Section~\ref{sc: simple_sec} that in this case there is no need to consider augmented independent sets: it is enough to apply the second moment method to \textit{maximal} independent sets.

    \item The upper bound on $p$ is not essential and is taken only for convenience when working with asymptotics, which have a simpler form in this regime. The concentration of the independence number for larger $p$, close to a constant, follows from Theorem~\ref{th:simple_sec} from Section~\ref{sc: simple_sec}. Thus, Theorem~\ref{th:main_res} is a consequence of the theorem above and Theorem~\ref{th:simple_sec}.

    \item Given this theorem, it is easy to verify that $s_z$ also satisfies formula~\eqref{main_fm} (see Lemma~\ref{lm:s_diff}). This implies the result of Theorem~\ref{th:asympt} (and for $p>\log^{-(k-1)}n$, the result follows from Theorem~\ref{th:simple_sec}): throughout this work, we only deal with potential independent sets whose size is of the form~\eqref{main_fm}.

\end{itemize}
\vspace{\baselineskip}

\textbf{Organization of the paper.}
The following three sections provide the proof of the main Theorem~\ref{th:main_tech} above: Section~\ref{sc:tech} presents the necessary technical estimates for asymptotics related to $s$ as well as calculations for the optimal values of parameters $s=s_z, r$ for the quantity $Z_{s, r}$; Sections~\ref{sc:first_moment} and~\ref{sc:second_moment} prove the upper and lower bounds of Theorem~\ref{th:main_tech} respectively. Section~\ref{sc: simple_sec} demonstrates how to prove the concentration of the independence number for sufficiently large $p$ by considering only maximal independent sets: we included this a) for completeness and b) to cover $p$ values close to a constant regime. Section~\ref{sc:anti} provides the proof of the absence of concentration for small $p$ --- Theorem~\ref{th:antic}. Finally, Section~\ref{sc:further} is devoted to discussing further questions.

\section{Preliminary calculations}
\label{sc:tech}
Throughout this and the next two sections, we assume that $p$ satisfies the conditions in Theorem~\ref{th:main_tech}.
\subsection{Binomial coefficients relations} We will quite often use the following equality

\begin{equation}
\label{binom_diff}
    {t+1 \choose k} - {t \choose k} = \frac{f(t)}{(k-1)!}
\end{equation}
which holds for any pair of natural $t\geq k$. Also, when applying the second moment method to consider pairs of augmented sets with a large intersection, we will need the following technical lower bound for the difference between two binomial coefficients:
\begin{lemma}
\label{lm:bin_coeff_diff}
    Let $\delta > 0$ and for two natural numbers $t \geq r \geq k$, the inequality $f(r) \geq (1-\delta)f(t)$ holds. Then
    $${t \choose k} - {r \choose k} \geq (t-r)(1-\delta)\frac{f(t)}{(k-1)!}.$$
\end{lemma}
\begin{proof}
 $${t \choose k} - {r \choose k} \stackrel{\eqref{binom_diff}}{=} \sum\limits_{i=r}^{t-1} \frac{f(i)}{(k-1)!} \geq (t-r)\frac{f(r)}{(k-1)!} \geq (t-r)(1-\delta)\frac{f(t)}{(k-1)!},$$
     where we used that $f(i)$ is increasing for $i\geq 2$.
\end{proof}

\subsection{Asymptotics for values of $s=s(n)$ that are of the form~\eqref{main_fm}}

First, we note that all asymptotics below hold in particular for values of $s$ close to $s_x$. Obviously, $s \xrightarrow{} \infty$ and $s=o(n)$. The main idea in defining the asymptotic relation~~\eqref{main_fm} was to satisfy the following property:
\begin{align*}
\frac{ne}{s} (1-p)^{\frac{f(s)}{k!}} &= (1+o(1))\frac{ne}{s} \text{exp}\left[-\log d + \frac{1}{k-1} \log \log d - \log \frac{e}{(k!)^{1/(k-1)}} + o(1)\right] \\
&= (1+o(1)) \frac{(\log d)^{\frac{1}{k-1}} (k!)^{\frac{1}{k-1}} }{p^{\frac{1}{k-1}} s} = (1+o(1)),
\end{align*}
from which we establish that
\begin{equation}
\label{1_minus_p}
    (1-p)^{\frac{f(s)}{k!}} = (1+o(1)) \frac{s}{ne}.
\end{equation}
But at the same time, one can similarly derive that
\begin{equation}
    \nonumber
    (1-p)^{O(s^{k-2})} = 1+o(1).
\end{equation}
Finally, two asymptotic relations follow, which we will use further on:
\begin{align}
     \left(1-(1-p)^{s \choose k-1}\right)^{n-s}  \stackrel{\eqref{1_minus_p}}{=} \left(1-(1+o(1))\left(\frac{s}{ne}\right)^{k}\right)^{n-s}&=e^{-(1+o(1))\frac{s^k}{e^kn^{k-1}}}, \label{pre_exp} \\
    \left[1-(1-p)^{s \choose k-1} - p{s \choose k-1}(1-p)^{{s \choose k-1}-1}\right]^{n-s} &= e^{-(1+o(1)) \frac{ps^{2k-1}}{(k-1)!e^k n^{k-1}}}.\label{pre_F}
\end{align}

\subsection{Defining and first estimating $s_z$ and $r_z$} To start, we define $E(n, k, s, r)$ as the expected value of the number of augmented independent sets of order $s$ with exactly $r$ edges in $H(n, k, p)$. There are ${n \choose s+r}$ ways to choose the location of our set, for each of which there are $\frac{(s+r)!}{(s-(k-1)r)! (k!)^r r!}$ ways to choose the edges of the matching inside. The probability of choosing internal edges accordingly is $p^r(1-p)^{{s+r \choose k} - r}$. Finally, the probability that each vertex outside our set has at least two edges drawn into this set is $F(n, k, s, r)^{n-s-r}$, where
 
\begin{equation}
    \label{F-formula}
    F(n, k, s, r) := 1 - (1-p)^{s+r \choose k-1} - {s+r \choose k-1}p(1-p)^{{s+r \choose k-1} - 1} .
\end{equation}
Thus, we have
\begin{equation}
    \label{E-formula}
    E(n, k, s, r) = {n \choose s+r} \frac{(s+r)!}{(s-(k-1)r)! (k!)^r r!} p^r(1-p)^{{s+r \choose k} - r} F(n, k, s, r)^{n-s-r}.
\end{equation}

\vspace{\baselineskip}
Now we define $s_z = s_z(n)$ as the largest $s \leq s_x$ for which there exists an $r$ such that $E(n, k, s, r) > n^{\varepsilon}$. The values $s_z$ and $s_x$ are close to each other, but not equal for small values of $p$, as the following lemma demonstrates:

\begin{lemma}
\label{lm:s_diff}
    If $p=\omega\left(\log n / n^{\frac{(k-1)^2}{k}}\right)$ then $s_z=s_x$. Otherwise, we have
    $$s_z \geq s_x - O\left(\frac{s_x^k}{n^{k-1}}\right) \geq s_x - o(s_x^{\frac{1}{k}}).$$
\end{lemma}
\begin{proof}
    Let us consider $r=0$ in~\eqref{E-formula}:
    $$E(n, k, s, 0) = {n \choose s} (1-p)^{{s \choose k}} F(n, k, s, 0)^{n-s}.$$
    In the case $p=\omega\left(n^{\frac{-(k-1)^2}{k}} \log n\right)$, the expression on the right-hand side of~\eqref{pre_F} is subpolynomial (i.e., this expression is $\exp[-o(\log n)]$). Therefore, $F(n, k, s, 0)^{n-s}$ is subpolynomial and $E(n, k, s_x, 0) > n^{\varepsilon}$, i.e., $s_z = s_x$.

    Next, we consider the case $p=O \left(n^{\frac{-(k-1)^2}{k}} \log n\right)$. Note that for $s=s_x - C\frac{s_x^{k}}{n^{k-1}}$, where $C$ is a sufficiently large constant, we have

    \begin{equation}
        \label{pre_ratio}
        \frac{{n \choose s}(1-p)^{s \choose k}}{{n \choose s+1}(1-p)^{s+1 \choose k}} \stackrel{\eqref{binom_diff}}{=} \frac{s+1}{n-s}(1-p)^{-\frac{f(s)}{(k-1)!}} \stackrel{\eqref{1_minus_p}}{=} (1+o(1)) \frac{n^{k-1}e^k}{s^{k-1}},
    \end{equation}
    and therefore
    \begin{align*}
        E(n, k, s, 0) &\stackrel{\eqref{pre_F}}{=} {n \choose s} (1-p)^{{s \choose k}} \text{exp}\Big\{({-(1+o(1)) \frac{ps^{2k-1}}{(k-1)!e^k n^{k-1}}}\Big\} \\
            &\text{\hspace{3pt}}\geq \left(\frac{n^{k-1}e^k}{2s^{k-1}}\right)^{\frac{Cs^k}{n^{k-1}}} \text{exp}\Big\{({-(1+o(1)) \frac{ps^{2k-1}}{(k-1)!e^k n^{k-1}}}\Big\}.
    \end{align*}
    Thus $E(n, k, s, 0) > n^{\varepsilon}$, if $C$ is a sufficiently large constant. Lemma~\ref{lm:s_diff} is proved.
\end{proof}
From this lemma, it particularly follows that estimates (\ref{1_minus_p})--(\ref{pre_F}) hold true on the entire interval $s_z \leq s \leq s_x$.
Now, for any $s$ in this interval, we define
$$r_M(s) = r_M(n, k, s) := \text{arg } \max_{r} \{E(n, k, s, r)\}.$$
Using this notation, we finally introduce formally the random variable $Z$. Let $r_z = r_M(s_z)$, then $Z$ is the number of augmented independent sets of order $s_z$ with exactly $r_z$ edges inside. Note that such sets have $s_z+r_z$ vertices and $\mathbb{E} [Z] = E(n, k, s_z, r_z) > n^{\varepsilon}.$

\subsection{Estimating of $r_z$}
We will estimate the ratio of successive terms of (\ref{E-formula}) when $r$ changes by $1$. Note the obvious observation concerning augmented sets: $r \leq s/(k-1)$; this fact will be used a couple of times below.

First, we show that $F(n, k, s, r)$ will have an insignificant effect on the ratio. From (\ref{1_minus_p}), one can easily verify that $F(n, k, s, r) = 1 - o(1)$. Now, let us estimate the difference between two consecutive values of $F$ in~\eqref{F-formula}:
\begin{align*}
\nonumber
        F(n, k, s, r+1) - F(n, k, s, r) &\leq (1-p)^{{s+r \choose k-1} - 1} \left(1+{s+r \choose k-1}p\right) \left(1 - (1-p)^{{s+(r+1) \choose k-1} - {s+r \choose k-1}}\right) \\
        &= (1+o(1)) (1-p)^{{s+r \choose k-1}} {s+r \choose k-1} \frac{(s+r)^{k-2}}{(k-2)!} p^2 \\
        &= (1+o(1)) \frac{(s+r)^{2k-3}}{(k-1)!(k-2)!} p^2 (1-p)^{{s+r \choose k-1}}.
\end{align*}
Therefore, we have
\begin{align*}
    \nonumber
         \frac{F(n, k, s, r+1)^{n-s-(r+1)}}{F(n, k, s, r)^{n-s-r}} &= \left(\frac{F(n, k, s, r+1)}{F(n, k, s, r)}\right)^{n-s-r} F(n, k, s, r+1)^{-1}  \\
         &\leq (1+o(1))\left(1 + \frac{ (1+o(1)) (s+r)^{2k-3}p^2 (1-p)^{{s+r \choose k-1}})}{(k-1)!(k-2)!F(n, k, s, r)}\right)^{n-s-r} \\
         &= (1+o(1)) \left(1+(1+o(1)) \frac{(s+r)^{2k-3}}{(k-1)!(k-2)!} p^2 (1-p)^{s+r \choose k-1}\right)^{n-s-r}.
\end{align*}
Since
\begin{align*}
    \frac{(s+r)^{2k-3}}{(k-1)!(k-2)!} p^2 (1-p)^{s+r \choose k-1} (n-s-r) &\text{\hspace{4pt}}\leq \frac{(ks)^{2k-3}}{(k-1)^{2k-3}(k-1)!(k-2)!} p^2n (1-p)^{{s \choose k-1}} \\
    &\stackrel{\eqref{1_minus_p}}{=} \frac{(1+o(1))k^{2k-3}}{(k-1)^{2k-3}(k-1)!(k-2)!e^k} \frac{s^{3k-3}p^2}{n^{k-1}} \\
    &\text{\hspace{4pt}}= o(1),
\end{align*}
then we get, as desired, that
\begin{equation}
    \label{F_ratio}
    \frac{F(n, k, s,r+1)^{n-s-r-1}}{F(n, k, s, r)^{n-s-r}} = (1+o(1)).
\end{equation}

With this bound on the ratio of two probabilities in hand, and applying the same machinery as in~\eqref{pre_ratio}, we get that under the condition $r = o\left(p^{-\frac{1}{k-1}}\right)$ (to use asymptotics above):
\begin{align}
    \frac{E(n, k, s, r+1)}{E(n, k, s, r)} &= (1+o(1)) \frac{{n \choose s+r+1} (1-p)^{s+r+1 \choose k}}{{n \choose s+r} (1-p)^{s+r \choose k}} \frac{(s+r+1)f(s-(k-1)r)p}{k!(r+1)} \nonumber \\
    & \xlongequal{\text{\eqref{binom_diff}, \eqref{1_minus_p}}} (1+o(1)) \left(\frac{s^{k-1}}{n^{k-1}e^{k}}\right) \frac{(s+r+1)f(s-(k-1)r)p}{k!(r+1)} \label{E_ratio1} \\
    &= (1+ o(1)) \frac{s^{2k-1}p}{k!e^{k}(r+1)n^{k-1}}. \label{E_ratio2}
\end{align}
This calculation provides the desired estimate for $r_z$:

\begin{lemma}
\label{lm:r_z}
    $$r_z = \left\lceil (1+o(1)) \frac{(k!)^{\frac{k}{k-1}}\log (d)^{\frac{2k-1}{k-1}}}{e^{k}n^{k-1}p^{\frac{k}{k-1}}} - 1 \right\rceil.$$
\end{lemma}

\begin{proof}
    First note that the form of~\eqref{E_ratio2} suggests $r_M(s)$ should be roughly $\frac{s^{2k-1}p}{n^{k-1}}$ as this is the regime in which the ratio is roughly one. Since $\frac{s^{2k-1}p}{n^{k-1}} = o\left(p^{\frac{1}{k-1}}\right)$, the estimate given by~\eqref{E_ratio2} is also true in the neighborhood of $r=\frac{s^{2k-1}p}{n^{k-1}}$. Moreover, for $r = \Omega\left(p^\frac{1}{k-1}\right)$ the arguments above can be easily adapted (using inequality~\eqref{1_minus_p_ineq}) to show that $\frac{E(n, k, s, r+1)}{E(n, k, s, r)}$ is bounded above by the expression in~\eqref{E_ratio2}. It follows that we have
    $$r \geq (1+o(1))\frac{s^{2k-1}p}{k!e^{k}n^{k-1}} \hspace{11pt} \Rightarrow \hspace{10pt}  \frac{E(n, k, s, r +1)}{E(n, k, s, r)} < 1.$$
    Therefore, we can conclude that
    \begin{equation}
    \label{r_M_estimate}
        r_M(s) =  \left\lceil (1+o(1))\frac{s^{2k-1}p}{k!e^{k}n^{k-1}} - 1\right\rceil = O\left(\frac{\log (d) s^{k}}{n^{k-1}} \right).
    \end{equation}
   In light of Lemma~\ref{lm:s_diff}, the proof is complete.
\end{proof}

We note in passing that Lemma~\ref{lm:r_z} implies that the described augmentation and $r_z$ become relevant at $p=\Theta\left(\log (n)^{\frac{2k-1}{k}}n^{-\frac{(k-1)^2}{k}}\right)$, as for $p$ larger than this regime $r_z = 0$. Furthermore, for all $p$ that we consider in this paper, i.e. for $p > n^{-\frac{k(k-1)}{k+1} + \varepsilon}$, we have
\begin{equation}
    \label{r_order}
    r_z = o(s_x^{\frac{1}{k}}).
\end{equation}

\subsection{Estimating of $s_z$} We may assume $p= O\left(\log (n) n^{-\frac{(k-1)^2}{k}}\right)$ as we have $s_z = s_x$ for larger $p$ by Lemma~\ref{lm:s_diff}. Note further that~\eqref{r_M_estimate} implies that $r_M(s) = \Omega \left( \log n\right)$ for $s_z \leq s \leq s_x$ and $p$ in this regime. Recall that, appealing to~\eqref{pre_F}, we have
$$E(n, k, s_x, 0) = \mathbb{E} [X_{s_x}] e^{-(1+o(1))\frac{ps_x^{2k-1}}{(k-1)!e^kn^{k-1}}}.$$
Next observe that, since $F(n, k, s+1, r) = F(n, k, s, r+1)$ by~\eqref{F-formula}, we can easily adapt the argument that proves~\eqref{F_ratio} to establish
$$\frac{F(n, k, s+1, r)^{n-s-r-1}}{F(n, k, s, r)^{n-s-r}} = 1+o(1)$$
for $s_z \leq s \leq s_x$. So we can calculate the ratio of expectations over a change in $s$ in the same way that we established~\eqref{E_ratio2}. Namely, for $s_z \leq s \leq s_x$ and $r = o\left(p^{\frac{1}{k-1}}\right)$, it is not difficult to verify that
\begin{equation}
    \label{E_ratio_s}
    \frac{E(n, k, s+1, r)}{E(n, k, s, r)} = (1+o(1)) \frac{s^{k-1}}{n^{k-1}e^k}.
\end{equation}
Using~\eqref{pre_F}, \eqref{E_ratio_s}, \eqref{E_ratio2} and \eqref{r_M_estimate}, we write:
\begin{align*}
    &E(n, k, s, r_M(s)) = E(n, k, s_x, 0) \prod\limits_{\ell=s}^{s_x-1} \frac{E(n, k, \ell, 0)}{E(n, k, \ell+1, 0)} \prod\limits_{r=0}^{r_M(s)-1} \frac{E(n, k, s, r+1)}{E(n, k, s, r)} \\
    &= \mathbb{E} [X_{s_x}] e^{-(1+o(1))\frac{s_x^{2k-1}p}{(k-1)!e^kn^{k-1}}} \left( \frac{e^k n^{k-1}}{s^{k-1}}(1+o(1))\right)^{s_x-s} \left(\frac{s^{2k-1}p}{k!e^{k}n^{k-1}}(1+o(1))\right)^{r_M(s)} \frac{1}{r_M(s)!} \\
    &=  \mathbb{E} [X_{s_x}] e^{-(1+o(1))\frac{s_x^{2k-1}p}{(k-1)!e^kn^{k-1}} +(1+o(1))\frac{s_x^{2k-1}p}{k!e^{k}n^{k-1}}} \left( \frac{e^k n^{k-1}}{s^{k-1}}(1+o(1))\right)^{s_x-s} \\
    &= \mathbb{E} [X_{s_x}] e^{-(1+o(1))\frac{(k-1)s_x^{2k-1}p}{k!e^kn^{k-1}}} \left( \frac{e^k n^{k-1}}{s^{k-1}}(1+o(1))\right)^{s_x-s}.
\end{align*}
This estimate is sufficient to prove the second part of the Theorem~\ref{th:main_tech}:
\begin{lemma}
\begin{align*}
n^{-\frac{k(k-1)}{k+1} + \varepsilon} <p \ll \log (n) n^{-\frac{(k-1)^2}{k}}   \Rightarrow  s_x - s_z = (1+o(1)) \frac{s^k}{e^kn^{k-1}} = (1+o(1)) \frac{(k!\log d)^{\frac{k}{k-1}}}{e^kn^{k-1} p^{\frac{k}{k-1}}}; \\
p = C \log (n) n^{-\frac{(k-1)^2}{k}}   \Rightarrow   s_x - s_z = \left\lceil \frac{((k-1)!)^{\frac{k}{k-1}}}{C^{\frac{k}{k-1}}e^k} - \frac{k \log \mathbb{E} [X_{s_x}]}{(k-1)\log n} + \frac{k}{k-1}\varepsilon \pm o(1)\right\rceil.
\end{align*}
\end{lemma}

\begin{proof} 
We note that for $p$ in these ranges we have
    $$\log \frac{n}{s} = (1+o(1)) \log d = (1+o(1)) \frac{s^{k-1}p}{k!}.$$
It follows that $s_x - s_z$ --- is the smallest integer $\mathcal{S}$ such that
    $$\log \mathbb{E} [X_{s_x}] - (1+o(1)) \frac{s_x^{2k-1}p}{k(k-2)!e^kn^{k-1}} + \mathcal{S}(1+o(1)) \frac{s^{k-1}p}{k(k-2)!} > \varepsilon \log n.$$
If $p \ll \log (n) n^{-\frac{(k-1)^2}{k}}$ then $\log \mathbb{E} [X_{s_x}] = O(\log n) = o(\frac{s_x^{2k-1}p}{n^{k-1}})$ and the first part of the Lemma follows. On the other hand, if $p = C \log (n) n^{-\frac{(k-1)^2}{k}}$ then $\frac{s_x^{2k-1}p}{n^{k-1}} = (1+o(1)) \frac{(k-1!)^{\frac{2k-1}{k-1}} \log n}{C^{\frac{k}{k-1}}}$ and $s^{k-1}p = (1+o(1)) (k-1)! \log n$ and we recover the second part of the Lemma.

\end{proof}

\section{Upper bound}
\label{sc:first_moment}
In this section we show that, with high probability, no augmented independent set of order $s$ appears for any $s_z + 2 \leq s \leq s_x+1$. This is sufficient because we can simply consider the expected number of independent sets of size $s_x + 2$ to rule out the appearance of any larger independent set. We emphasize again that we only consider values of $s$ satisfying~\eqref{main_fm}. We begin with a simple observation:
\begin{lemma}
\label{lm:E_sum}
    For all $s$ that satisfy~\eqref{main_fm} we have
    $$\sum\limits_{r=0}^{s/(k-1)} E(n, k, s, r) \leq 4 \sum\limits_{r=0}^{2r_M(s)} E(n, k, s, r).$$
\end{lemma}
\begin{proof}
    We start by noting that, similarly to \eqref{E_ratio2}, one can show that
    $\frac{E(n, k, s, r+1)}{E(n, k, s, r)} \leq (1+o(1)) \frac{s^{2k-1}p}{k!e^{k}n^{k-1}(r+1)}.$
Now we consider cases depending on the value of $r_M = r_M(s)$. First, note that if $r_M=0$ then by ~\eqref{r_M_estimate} we get $\frac{s^{2k-1}p}{k!e^{k}n^{k-1}} < 1 + o(1)$ and
$$\sum\limits_{r=0}^{s/(k-1)} E(n, k, s, r) \leq (1+o(1))E(n, k, s, 0) \sum\limits_{r=0}^{\infty} \frac{1}{r!} \leq 4E(n, k, s, 0).$$
Now suppose $r_M \geq 1$. Again recalling~\eqref{r_M_estimate}, note that if $r > 2r_M$ then we have
$$E(n, k, s, r+1)/E(n, k, s, r) \leq \frac{2}{3}.$$
So we get
\begin{align*}
    \sum\limits_{r=0}^{s/(k-1)} E(n, k, s, r) &= \sum\limits_{r=0}^{2r_M}E(n, k, s, r) + \sum\limits_{r=2r_M+1}^{s/(k-1)} E(n, k, s, r) \\
    &\leq\left(\sum\limits_{r=0}^{2r_M}E(n, k, s, r)\right) + 3E(n, k, s, 2r_M+1)  \\
    &\leq \left(\sum\limits_{r=0}^{2r_M}E(n, k, s, r)\right) + 3E(n, k, s, 2r_M).
\end{align*}
\end{proof}
Applying this lemma, we have
\begin{align*}
    \sum\limits_{r=0}^{(s+1)/(k-1)} E(n, k, s+1, r) &\stackrel{\text{L.}\ref{lm:E_sum}}{<}4\sum\limits_{r=0}^{2r_M(s+1)} E(n, k, s+1, r) \\ &\stackrel{\eqref{E_ratio_s}}{<} \frac{s^{k-1}}{n^{k-1}} \sum\limits_{r=0}^{2r_M(s+1)} E(n, k, s, r)  < \frac{s^{k-1}}{n^{k-1}}\sum\limits_{r=0}^{s / (k-1)} E(n, k, s, r).
\end{align*}

Therefore, we get
\begin{align*}
    \sum\limits_{s=s_z+2}^{s_x+1} \left(\sum\limits_{r=0}^{s/(k-1)} E(n, k, s, r)\right) &\text{\hspace{4pt}}<\text{\hspace{3pt}} \frac{2s_z^{k-1}}{n^{k-1}}\sum\limits_{r=0}^{(s_z+1)/(k-1)} E(n, k, s_z+1, r) \\
    &\stackrel{\text{L.} \ref{lm:E_sum}}{\leq} \frac{8s_z^{k-1}}{n^{k-1}}\sum\limits_{r=0}^{2r_M(s_z+1)} E(n, k, s_z+1, r) \\
    &\stackrel{\eqref{r_M_estimate}}{\leq} \frac{8s_z^{k-1}}{n^{k-1}} \frac{3s_z^{2k-1}p}{e^{k}n^{k-1}} E(n, k, s_z+1, r_M(s_z+1)) \\
    &\stackrel{\hspace{20pt}}{\leq} \frac{8s_z^{k-1}}{n^{k-1}} \frac{3s_z^{2k-1}p}{e^{k}n^{k-1}} n^{\varepsilon} \\ &\stackrel{\hspace{20pt}}{=} O\left( \frac{s_z^{3k-2}pn^{\varepsilon}}{n^{2k-2}}\right) = o(1).
\end{align*}
Hence, whp no augmented independent set of order $s$ for any $s \geq s_z+2$ appears, and by Lemma~\ref{lm:idea} we have $\alpha(H) \leq k_z+1$ whp.
\section{Lower bound: second moment of $Z$}
\label{sc:second_moment}

In this section we apply the second moment method to show that an augmented independent set of order $s_z$ appears with high probability; more specifically, we prove that such a set with $s_z+r_z$ vertices and $r_z=r_M(s_z)$ internal edges appears whp. For convenience throughout the rest of this section we let
$$s=s_z, \quad \quad r=r_z=r_M(s_z) \quad \text{and}\quad \tilde{s} = s+r.$$

Recall that the random variable $Z$ counts the number of augmented independent sets of order $s$ with $r$ internal edges, and by definition of $r_M$ and $Z$ we have
$$\mathbb{E} [Z] > n^{\varepsilon},$$
and our goal is to show that $Z > 0$ whp.

Break up $Z$ into a sum of indicator random variables. Let $\mathbb{T}$ be the collection of all pairs $(T, m_T)$ where $T$ is a set of vertices of size $\tilde{s}$ and $m_T$ is $r$ pairwise disjoint internal edges of $T$. Note that $$|\mathbb{T}| = {n \choose \tilde{s}} \frac{\tilde{s}!}{(\tilde{s}-kr)!(k!)^r r!}.$$
For each pair $(T, m_T) \in \mathbb{T}$, we denote by $Z_{T, m_T}$ the indicator random variable for the event that $T$ is an augmented independent set with matching $m_T$ inside. We have
\begin{align*}
    {\sf Var}(Z) &\leq \sum\limits_{(T, m_T) \in \mathbb{T}} {\sf Var}(Z_{T, m_T}) + \sum\limits_{T \neq U} {\sf Cov}(Z_{T, m_T}, Z_{U, m_U})  \\
    &\leq \mathbb{E}(Z) + \sum\limits_{T \neq U} [\mathbb{E} (Z_{T, m_T}Z_{U, m_U}) - \mathbb{E} (Z_{T, m_T}) \mathbb{E} (Z_{U, m_U}) ],
\end{align*}
where here and throughout this section such summations are over all $(T, m_T), (U, m_U) \in \mathbb{T}$ that have the specified $T$ and $U$. Note that we are making use of the full covariance we will see below that it is necessary to handle sets $T, U$ such that $|T \cap U|$ has order $\frac{s^2}{n}$. Next, let
\begin{align}
    h_i &:= \frac{1}{\mathbb{E} (Z)^2} \sum\limits_{|T \cap U| = i} \mathbb{E} (Z_{T, m_T}Z_{U, m_U}) - \mathbb{E} (Z_{T, m_T}) \mathbb{E} (Z_{U, m_U}) \label{h_def} \\
    &= \left(\frac{1}{|\mathbb{T}|}\right)^2 \sum\limits_{|T \cap U| = i} \left(\frac{\mathbb{E} (Z_{U, m_U} | Z_{T, m_T}=1)}{\mathbb{E} (Z_{T, m_T})} - 1\right). \label{h_formula}
\end{align}
Applying the second moment method we get $\mathbb{P} (Z=0) \leq {\sf Var}[Z] /\mathbb{E} [Z]^2$. Therefore, our goal is to check the following condition:
$$\sum\limits_{i=0}^{\tilde{s}-1} h_i = o(1).$$
We emphasize that $\mathbb{E} (Z)$ can be significantly different from $\mathbb{E} (X)$ here. We consider three ranges of $i$, for each of which we will establish the desired bounds for the sum of $h_i$:

\subsection{Case $1$: $i < \left(\frac{1}{\log(n)p}\right)^{\frac{1}{k-1}}$}

Fix a pair $(T, m_T) \in \mathbb{T}$ and define event $\mathcal{E} := \{Z_{T, m_T} = 1\}$. In other words, $\mathcal{E}$ is the event that $T$ gives an augmented independent set with the prescribed matching $m_T$. Let also $\mathcal{E}_U := \{Z_{U, m_U} = 1\}$. Recalling~\eqref{h_formula}, we seek to bound
\begin{equation}
    \nonumber
    h_i = \frac{1}{|\mathbb{T}|} \sum\limits_{U: |T \cap U|=i} \left(\frac{\mathbb{P} (\mathcal{E}_U \text{ }|\text{ } \mathcal{E})}{\mathbb{P} (\mathcal{E})} - 1\right).
\end{equation}
Now, for $\ell=0, 1, \ldots, r$ let $\mathbb{T}_{\ell}$ be the collection of pairs $(U, m_U)$ with the property that $m_T$ and $m_U$ have exactly $\ell$ edges in common. Also define

$$B = \left\lceil \frac{1}{p^{\frac{1}{k}} \log (n)^{\frac{1}{k}}} \right\rceil,$$
and, for $B < i < B^{\frac{k}{k-1}}$, let $\mathbb{T}_{i, \ell}$ be the collection of pairs $(U, m_U) \in \mathbb{T_{\ell}}$ such that $|T \cap U| = i$. Finally, let $\mathbb{T}'_{\ell}$ be the collection of pairs $(U, m_U) \in \mathbb{T}_{\ell}$ with the additional property $|T \cap U| \leq B$. We make two auxiliary claims:
\begin{lemma}
\label{lm:prob_relations}
    If $(U, m_U) \in \mathbb{T}_{i, \ell}$ and $i < \left(\frac{1}{\log (n)p}\right)^{1/(k-1)}$ then
    $$\mathbb{P} (\mathcal{E}_U\text{ } |\text{ } \mathcal{E}) \leq  \exp\{O(p^2s^{2k-1}i^{k-1}/n^{k-1})+o(1)\} \frac{\mathbb{P}(\mathcal{E})}{p^{\ell}(1-p)^{i \choose k}}.$$
\end{lemma}
\begin{proof}
    In order to consider the conditioning on $\mathcal{E}$, we define an additional parameter $w$ as the number of edges in $m_U$ that have exactly one vertex out of $T \cap U$. And let $W$ the set of vertices in $U \backslash T$ that are included in such edges:
    $$W = \{v \in U \backslash T: \exists e\in E(H) \cap m_U \text{ such that } e \backslash T=v\}.$$
    For each vertex $v \in W$ let $e(v)$ be the edge of $m_U$ that contains $v$.
    
    We condition as follows. First, we simply condition on all $k$-vertex subsets within $T$ appearing as edges and non-edges as required for the event $\mathcal{E}$. But we proceed more carefully for vertices $v$ that are not in $T$. For such vertices we must reveal information about the edges from $v$ into $T$. We do this by observing whether or not the potential edges of the form $e=(v, u_1, u_2, \ldots, u_{k-1})$, where $u_1, u_2, \ldots, u_{k-1}  \in T$, actually are hypergraph edges one at a time, starting with potential edges containing vertices from $T \backslash U$. And we stop as soon as we observe at least two such edges. For vertices $v \in W$ we also add the restriction that the edge $e(v) \in m_U$ is the last edge observed in this process. Thus, for a vertex $v \in U \backslash T$, this process either observes two edges containing $v$ and at least one vertex from $T \backslash U$ --- and therefore observes no potential edges inside $U$ that contain $v$ or the process observes at most one edge intersecting $T \backslash U$ and therefore observes edges between $v$ and $T\cap U$ and reveals that at least one such edge appears. In the latter case the conditional probability of $\mathcal{E}_U$ is zero unless this vertex $v$ is in $W$ and the observed edge is $e(v)$.
    
    For each set $R \subset U \backslash T$ we define $\mathcal{F}_R$ to be the event that $\mathcal{E}$ holds and $R$ is the set of vertices $v \in U \backslash T$ with the property that our process reveals edges inside $U$. Note that we have
    $$\mathcal{E} = \bigcup\limits_{R \subseteq U\backslash T} \mathcal{F}_R, \text{ and }$$
    $$R \nsubseteq W \quad \Rightarrow \quad \mathbb{P}(\mathcal{E}_U\cap \mathcal{F}_R)=0.$$

    Next, let $I = T \cap U$. We define $\mathcal{I}$ to be the event that the process detailed above resorts to revealing potential edges between $v$ and $I$ for more than $\log^2 (n) s$ vertices $v \notin T$. Note that the probability that a particular vertex $v \notin T$ resorts to observing potential edges between $v$ and $I$ is
    $$1 - \frac{1-(1-p)^{{\tilde{s} \choose k-1} - {i \choose k-1}} - p\left({\tilde{s} \choose k-1} - {i \choose k-1}\right)(1-p)^{{\tilde{s} \choose k-1} - {i \choose k-1}-1}}{1-(1-p)^{\tilde{s} \choose k-1} - p{\tilde{s} \choose k-1}(1-p)^{{\tilde{s} \choose k-1} - 1}} \stackrel{\eqref{1_minus_p}}{=} O\left(\frac{(\log n)s^k}{n^k}\right),$$
    where we again used the fact that $F(n, k, s, r) = 1-o(1)$. Thus, the expected number of vertices that observe edges to $I$ is $O\left(\frac{(\log n) s^k}{n^{k-1}}\right) = o(s)$. As these events are independent, the Chernoff bound (see ~\cite[Corollary 2.4]{janson}) then implies
    \begin{equation}
    \label{prob_i}
        \mathbb{P}(\mathcal{I}) \leq \exp\{-\log^2 (n) s\}.
    \end{equation}
    
    Finally, we are ready to put everything together. Applying the law of total probability, we have
    \begin{align*}
    \mathbb{P}(\mathcal{E}_U\text{ } |\text{ } \mathcal{E}) &= \sum\limits_{R \subseteq U \backslash T}\left(\mathbb{P}(\mathcal{E}_U\text{ } |\text{ } \mathcal{F}_R \cap \mathcal{I}) \cdot \frac{\mathbb{P}(\mathcal{F}_R \cap \mathcal{I})}{\mathbb{P}(\mathcal{E})} + \mathbb{P}(\mathcal{E}_U \text{ }| \text{ }\mathcal{F}_R \cap \overline{\mathcal{I}}) \cdot \frac{\mathbb{P}(\mathcal{F}_R \cap \overline{\mathcal{I}})}{\mathbb{P}(\mathcal{E})}\right) \\
    &\leq \frac{\mathbb{P}(\mathcal{I})}{\mathbb{P}(\mathcal{E})} +  \sum\limits_{R \subseteq U \backslash T}\mathbb{P}(\mathcal{E}_U\text{ } |\text{ } \mathcal{F}_R \cap \overline{\mathcal{I}}) \frac{\mathbb{P}(\mathcal{F}_R)}{\mathbb{P}(\mathcal{E})} \\
    &= \frac{\mathbb{P}(\mathcal{I})}{\mathbb{P}(\mathcal{E})} +  \sum\limits_{R \subseteq W}\mathbb{P}(\mathcal{E}_U\text{ } |\text{ } \mathcal{F}_R \cap \overline{\mathcal{I}}) \frac{\mathbb{P}(\mathcal{F}_R)}{\mathbb{P}(\mathcal{E})}.
    \end{align*}
   Let us estimate the factors of the sum's terms. Provided that $R \subseteq W$:
    \begin{align*}
        \frac{\mathbb{P}(\mathcal{F}_R)}{\mathbb{P}(\mathcal{E})} &\leq \left(\frac{(1-p)^{{\tilde{s} \choose k-1} - {i \choose k-1}} + p\left({\tilde{s} \choose k-1} - {i \choose k-1}\right)(1-p)^{{\tilde{s} \choose k-1} - {i \choose k-1}-1}}{1-(1-p)^{\tilde{s} \choose k-1} - p{\tilde{s} \choose k-1}(1-p)^{{\tilde{s} \choose k-1} - 1}}\right)^{|R|} \\
        &= \left((1+o(1)) \frac{ps^{k+1}}{e^kn^k}\right)^{|R|},
    \end{align*}
    \begin{align*}
        \mathbb{P} (\mathcal{E_U} \text{ } |\text{ }  \mathcal{F}_R \cap \overline{\mathcal{I}}) &\leq {i \choose k-1}^{-|R|}(1-p)^{{\tilde{s} \choose k} - {i \choose k} - {i \choose k-1}|R| - (r-|R|-\ell)}p^{r-|R|-\ell} \times \\ &\quad\quad\quad\quad\quad \times\left(1-(1-p)^{{\tilde{s} \choose k-1}} - p{{\tilde{s} \choose k-1}}(1-p)^{{{\tilde{s} \choose k-1}}-1}\right)^{n-2\tilde{s}+i-\log (n)^2 s} \\
        &\leq \frac{\mathbb{P}(\mathcal{E})}{p^{\ell}(1-p)^{i \choose k}} \left({i \choose k-1}p\right)^{-|R|} \exp\left\{2p{i \choose k-1}r\right\} \times \\
        &\quad\quad\quad\quad\quad \times \exp\left\{O\left(\frac{s^{2k-1}p}{n^k}\right)\left(\tilde{s}+\log (n)^2 s \right)\right\} \\
        &\leq (1+o(1)) \frac{\mathbb{P}(\mathcal{E})}{p^{\ell}(1-p)^{i \choose k}} \left({i \choose k-1}p\right)^{-|R|} \exp\left\{2p{i \choose k-1}r\right\},
    \end{align*}
where we use $|R| \leq r$ and $s^{2k}p < n^{k-\varepsilon}$. Thus
\begin{align*}
    \sum\limits_{R \subseteq W}&\mathbb{P}(\mathcal{E}_U\text{ } |\text{ } \mathcal{F}_R \cap \overline{\mathcal{I}}) \frac{\mathbb{P}(\mathcal{F}_R)}{\mathbb{P}(\mathcal{E})} \\
    &\stackrel{\eqref{r_M_estimate}}{\leq} \frac{(1+o(1)) \exp\left\{O\left(p^2i^{k-1}s^{2k-1} /n^{k-1}\right)\right\} \mathbb{P}(\mathcal{E})}{p^{\ell}(1-p)^{i \choose k}} \sum\limits_{j=0}^w {w \choose j} \left((1+o(1))\frac{s^{k+1}}{e^kn^k{i \choose k-1}}\right)^j \\
    &\text{\hspace{5pt}}\leq\text{\hspace{6pt}} \frac{(1+o(1)) \exp\left\{O\left(p^2i^{k-1}s^{2k-1} /n^{k-1}\right)\right\} \mathbb{P}(\mathcal{E})}{p^{\ell}(1-p)^{i \choose k}},
\end{align*}
where we use $w \leq i$ to bound the sum.
Finally, using~\eqref{prob_i}, we get:
$$\frac{\mathbb{P}(\mathcal{I})}{\mathbb{P}(\mathcal{E})^2} \leq \exp\{-\log^2 (n) s + O(\log (n)s)\} = o(1),$$
which completes the proof of the Lemma.
\end{proof}
\begin{lemma}
\label{lm:t_relations}
For $\ell = 1, 2, \ldots, r$ we have
$$|\mathbb{T}_{\ell}| \leq |\mathbb{T}| \left(\frac{2k!r^2}{n^k}\right)^{\ell};$$
furthemore, we have
$$|\mathbb{T}_{\ell, i}| \leq |\mathbb{T}| \left(\frac{2k!r^2i^k}{\tilde{s}^{2k}}\right)^{\ell} \left(\frac{e\tilde{s}^2}{ni}\right)^i.$$
\end{lemma}
\begin{proof}
    $$|\mathbb{T}_{\ell}| \leq {r \choose \ell}{n - k\ell \choose \tilde{s} - k\ell}\frac{(\tilde{s}-k\ell)!}{(\tilde{s}-kr)!(k!)^{r-\ell}(r-\ell)!},$$
therefore
$$\frac{|\mathbb{T}_{\ell}|}{|\mathbb{T}|} \leq {r \choose \ell} \frac{\tilde{s}(\tilde{s}-1)\ldots (\tilde{s}-k\ell+1)}{n(n-1)\ldots (n-k\ell+1)} \times \frac{(k!)^{\ell} \times r(r-1)\ldots (r-\ell+1)}{\tilde{s}(\tilde{s}-1)\ldots (\tilde{s}-k\ell+1)} \leq \left(\frac{2k!r^2}{ n^k}\right)^{\ell}.$$

The calculation of the second part of the Lemma is similar.
$$|\mathbb{T}_{\ell, i}| \leq {r \choose \ell} {\tilde{s}-k\ell \choose i-k\ell}{n-\tilde{s} \choose \tilde{s}-i} \frac{(\tilde{s}-k\ell)!}{(\tilde{s}-kr)!(k!)^{r-\ell}(r-\ell)!},$$
therefore
\begin{align*}
    \frac{|\mathbb{T}_{\ell, i}|}{|\mathbb{T}|} &\leq \frac{{r \choose \ell} {\tilde{s}-k\ell \choose i-k\ell}{n-\tilde{s} \choose \tilde{s}-i} {\tilde{s} \choose i}}{{n \choose i}{n-i \choose \tilde{s}-i}}\times \frac{(k!)^{\ell} \times r(r-1)\ldots (r-\ell+1)}{\tilde{s}(\tilde{s}-1)\ldots (\tilde{s}-k\ell+1)} \\
   & \leq {r \choose \ell} \left(\frac{e\tilde{s}^2}{ni}\right)^i \left(\frac{i}{\tilde{s}}\right)^{k\ell} \times \frac{(k!)^{\ell} \times r(r-1)\ldots (r-\ell+1)}{\tilde{s}(\tilde{s}-1)\ldots (\tilde{s}-k\ell+1)} \\
   &\leq \left(\frac{2k!r^2i^k}{\tilde{s}^{2k}}\right)^{\ell} \left(\frac{e\tilde{s}^2}{ni}\right)^i.
\end{align*}
\end{proof}
We are now ready to check $\sum\limits_{i=0}^{1/\sqrt[k-1]{p \log n}} h_i = o(1)$ using Lemmas~\ref{lm:prob_relations} and~\ref{lm:t_relations}. We will split this into two sub-ranges. \\

{\bf{Subcase 1.1: ${i \leq B}$}}. In this case by Lemma~\ref{lm:prob_relations} we have $\mathbb{P} (\mathcal{E}_U \text{ } | \text{ } \mathcal{E}) \leq  (1+o(1)) \mathbb{P}(\mathcal{E}) p^{-\ell}$, hence:
\begin{align*}
    \sum\limits_{i=0}^B h_i &\hspace{5pt}\leq\hspace{4pt} \frac{1}{|\mathbb{T}|} \sum\limits_{\ell=0}^r \sum\limits_{(U, m_U) \in \mathbb{T}_{\ell}'} \left(\frac{\mathbb{P} (\mathcal{E}_U \text{ } | \text{ } \mathcal{E})}{\mathbb{P}(\mathcal{E})}-1\right) \\
    &\stackrel{\text{L.}\ref{lm:t_relations}}{\leq} \frac{o(|\mathbb{T}_0'|)}{|\mathbb{T}|} + \sum\limits_{\ell=1}^r \left(\frac{3k!r^2}{pn^k}\right)^{\ell} \\
    &\stackrel{\eqref{r_M_estimate}}{\leq} o(1)+\sum\limits_{r=1}^{\ell} \left(n^{-\frac{2k-2}{k+1}}\right)^{\ell} \\
    &\hspace{6pt}=\hspace{6pt}o(1).
\end{align*}
{\bf{Subcase 1.2: $B < i < B^{\frac{k}{k-1}}$}}.
\begin{align*}
    h_i &\text{\hspace{5pt}}= \hspace{4pt} \frac{1}{|\mathbb{T}|} \sum\limits_{\ell=0}^r \sum\limits_{(U, m_U) \in \mathbb{T}_{\ell, i}} \left(\frac{\mathbb{P}(\mathcal{E}_U \text{ } | \text{ } \mathcal{E})}{\mathbb{P} (\mathcal{E})} - 1\right) \\
    &\stackrel{\text{L.}\ref{lm:prob_relations}}{\leq} \sum\limits_{\ell=0}^r \left(\frac{|\mathbb{T}_{\ell, i}|}{|\mathbb{T}|}\right) \frac{\exp\{O(p^2s^{2k-1}i^{k-1}/n^{k-1})+o(1)\}}{p^{\ell}(1-p)^{i \choose k}} \\
    &\stackrel{\text{L.}\ref{lm:t_relations}}{\leq}\sum\limits_{\ell=0}^r \left(\frac{2k!r^2i^k}{\tilde{s}^{2k}}\right)^{\ell} \left(\frac{e\tilde{s}^2}{ni}\right)^i\left(\frac{2\exp\{O(p^2s^{2k-1}i^{k-1}/n^{k-1})\}}{p^{\ell}(1-p)^{i \choose k}}\right) \\
    &\text{\hspace{5pt}}\leq \hspace{4pt}\left(\frac{e^2\tilde{s}^2}{ni(1-p)^{f(i-1)/k!}}\right)^i \sum\limits_{\ell=0}^r \left(\frac{2k!r^2}{\tilde{s}^{k}p}\right)^{\ell} \\
    &\text{\hspace{5pt}}\leq \hspace{4pt} \left(\frac{e^3\tilde{s}^2 p^{1/k}(\log (n))^{1/k}}{n}\right)^i \sum\limits_{\ell=0}^r \left(\frac{2k!r^2}{\tilde{s}^{k}p}\right)^{\ell}.
\end{align*}
Since $\frac{\tilde{s}^2p^{\frac{1}{k}}\log(n)^{1/k}}{n} = o(1)$ and $\frac{r^2}{\tilde{s}^kp} \stackrel{\eqref{r_M_estimate}}{=} O\left(\frac{\tilde{s}^{3k-2}p}{n^{2k-2}}\right) = o(1)$, we have
$$\sum\limits_{i=B}^{B^{\frac{k}{k-1}}} h_i = o(1).$$

\subsection{Case 2: $ \left(\frac{1}{\log(n)p}\right)^{\frac{1}{k-1}} \leq i <  f^{-1}((1-\varepsilon)f(\tilde{s}))$}

This case is the simplest to analyze, as we do not need to use the condition of vertices outside $T$ having at least two neighbors in $T$. We begin by bounding $\frac{1}{\mathbb{E}(Z_{T, m_T})^2} \sum\limits_{m_T, m_U} \mathbb{E}(Z_{T, m_T} Z_{U, m_U})$ for a fixed pair of sets $T, U$ intersecting in exactly $i$ vertices. If $T \cap U$ contains exactly $\ell$ edges, a bound for the number of ways to choose $m_T$ and $m_U$ is $\tilde{s}^{2kr-k\ell}$ (we have $2r-\ell$ edges total, each has less than $\tilde{s}^{k}$ choices). Thus, we have
\begin{equation}
\begin{gathered}
    \nonumber
    \sum\limits_{m_T, m_U} \mathbb{E}(Z_{T, m_T} Z_{U, m_U}) \leq \sum\limits_{\ell=0}^r \tilde{s}^{2kr-k\ell} p^{2r-\ell}(1-p)^{2{\tilde{s} \choose k} - {i \choose k} - 2r+\ell} \\
    = \tilde{s}^{2kr} p^{2r} (1-p)^{2{\tilde{s} \choose k} - {i \choose k} - 2r} \sum\limits_{\ell=0}^r \left(\frac{1-p}{\tilde{s}^kp}\right)^{\ell} = (1+o(1)) \tilde{s}^{2kr} p^{2r} (1-p)^{2{\tilde{s} \choose k} - {i \choose k} - 2r}.
\end{gathered}
\end{equation}
Therefore, we can write the following bound:
\begin{align*}
    \sum\limits_{m_T, m_U} \frac{\mathbb{E}(Z_{T, m_T} Z_{U, m_U})}{\mathbb{E}(Z_{T, m_T})^2} &\hspace{2pt}\stackrel{\eqref{pre_F}}{\leq} (1+o(1)) \tilde{s}^{2kr} (1-p)^{-{i \choose k}} e^{(2+o(1)) \frac{ps^{2k-1}}{(k-1)!e^kn^{k-1}}} \\
    &\stackrel{\eqref{r_M_estimate}}{\leq} (1-p)^{-{i \choose k}} e^{O(\log (n)r+1)}.
\end{align*}
It follows that
\begin{align*}
    h_i &\text{\hspace{3pt}}\leq \hspace{3pt}\frac{{\tilde{s} \choose i}{n-\tilde{s} \choose \tilde{s}-i}}{{n \choose \tilde{s}}} (1-p)^{-{i \choose k}} e^{O(\log (n)r+1)} \\
    &\text{\hspace{3pt}}\leq \hspace{3pt}\left(\frac{e\tilde{s}}{i}\right)^i \left(\frac{\tilde{s}}{n}\right)^i (1-p)^{-{i \choose k}} e^{O(\log (n)r+1)} \\
    &\text{\hspace{3pt}}\leq \hspace{3pt}\left[\frac{e\tilde{s}^2}{in} \left((1-p)^{-\frac{f(\tilde{s})}{k!}}\right)^{f(i-1)/f(\tilde{s})}\right]^i e^{O(\log (n)r+1)} \\ &\stackrel{\eqref{1_minus_p}}{\leq} \left[(1+o(1))\frac{e\tilde{s}^2 \log (n)^{1/(k-1)}p^{1/(k-1)}}{n} \left(\frac{ne}{\tilde{s}}\right)^{1-\varepsilon}\right]^i e^{O(\log (n)r+1)} \\
    &\text{\hspace{3pt}}= \hspace{3pt}\left[(1+o(1))\frac{e^{2-\varepsilon}\tilde{s}^{1+\varepsilon} \log (n)^{1/(k-1)}p^{1/(k-1)}}{n^{\varepsilon}}\right]^i e^{O(\log (n)r+1)}.
\end{align*}
As $r = o(s^{1/k}) = o(i)$ by~\eqref{r_order} and $\tilde{s}^{1+\varepsilon} \log(n)^{1/(k-1)} p^{1/(k-1)} n^{o(1)} = o(n^{\varepsilon})$, we have
$$\sum\limits_{i=(p \log n)^{-\frac{1}{k-1}}}^{f^{-1}((1-\varepsilon)f(\tilde{s}))} h_i = o(1).$$

\subsection{Case 3: $i \geq f^{-1}((1-\varepsilon)f(\tilde{s}))$}

We provide an upper bound for $h_i$ based on the general form in~\eqref{h_def}. For simplicity, let $j=\tilde{s}-i$. For some fixed $i$, there are exactly ${n \choose \tilde{s}}{n-\tilde{s} \choose j}{\tilde{s} \choose j}$ ways to choose $T$ and $U$. Now fix sets $T$ and $U$.

We now bound the number of ways to choose the matching $m_T$ and $m_U$ that are compatible with the event $\{Z_{T, m_T}Z_{U, m_U}=1\}$. Let $Q_I$ be the set of matching edges contained in $T \cap U$, and let $q_I := |Q_I|$ (so $q_I$ is the same as $\ell$ from previous cases). Secondly, let $Q_O$ be  the set of edges in $m_T$ that are not in $Q_I$, and let $Q_O'$ be the set of edges of $m_U$ that are not in $Q_I$, and define $q_O := |Q_O|=|Q_O'|$ ( `I' stands for `inner', and `O' stands for `outer'). From the definitions, it is clear that $r=q_I+q_O$. It is also obvious that the number of ways to choose edges in $Q_O$ is bounded by
$$2^j \tilde{s}^{(k-1)q_O},$$
since we can first choose a leader vertex in $T\backslash U$ for each edge, and then complete each leader to a complete edge in at most $\tilde{s}^{k-1}$ ways.
By the symmetry reason this same expression estimates the number of ways to choose
edges in $Q_O'$.

Next, the number of ways to choose edges in $Q_I$ is bounded by $\tilde{s}^{kq_O}/((k!)^{q_I}q_I!)$. Putting these together, the number of ways to choose $m_T$ adn $m_U$ is at most
\begin{equation}
    \label{m_choices}
    \frac{4^j\tilde{s}^{(2k-2)q_O+kq_O}}{(k!)^{q_I}q_I!} \leq \frac{4^j\tilde{s}^{kr+(k-2)q_O} (k!r)^{q_O}}{(k!)^r r!}
\end{equation}
(assuming that $r^{q_O} = 1$ if $r=q_O=0$).

Now we estimate the probability of the event $\{Z_{T, m_T}Z_{U, m_U} = 1\}$. The probability the edges within $T$ and within $U$ are chosen accordingly is 
$$p^{r+q_O}(1-p)^{2{\tilde{s} \choose k} - {i \choose k}-r-q_O}.$$
Now consider vertices in $T\backslash U$ and $U \backslash T$. Recall that from a vertex $v \in T \backslash U$, there are at least two edges leading into $U$ (i.e., intersecting $U$ in $k-1$ vertices). So if $v$ already has an edge in $m_T$ leading into $U$, then there is at least one edge from $v$ intersecting $U\backslash T$ (since $v$ contained in at most one edge of $m_T$). Such an edge appears with probability at most $j{\tilde{s} \choose k-2}p$ by the union bound. If a vertex $v \in T\backslash U$ is not in an edge of $m_T$ going inside $U$, then it is  incident with at least two edges intersecting $U \backslash T$. The probability of these two edges appearing is at most $j^2{\tilde{s} \choose k-2}^2p^2$ by the union bound. Since the number of vertices in the first category is at most $q_O$, the total probability that edges from $T \backslash U$ into $U$ being chosen accordingly and vice versa is at most $\left(j{\tilde{s} \choose k-2}p\right)^{2j-q_O}$.
Hence, the probability that the edges within $T\cup U$ appear in accordance with the event $\{Z_{T, m_T}Z_{U, m_U} = 1\}$ is at most
$$p^{r+q_O}(1-p)^{2{\tilde{s} \choose k} - {i \choose k}-r-q_O} \left(j{\tilde{s} \choose k-2}p\right)^{2j-q_O}.$$
Finally, the probability that all vertices outside $T \cup U$ have at least two neighbours in $T$ is at most $\left(1-(1-p)^{\tilde{s} \choose k-1} - p{\tilde{s} \choose k-1}(1-p)^{{\tilde{s} \choose k-1}-1}\right)^{n-2\tilde{s}}$.
Note that, since
\begin{align*}
    \left(1-(1-p)^{\tilde{s} \choose k-1} - p{\tilde{s} \choose k-1}(1-p)^{{\tilde{s} \choose k-1}-1}\right)^{-s} &\text{\hspace{4pt}}\leq \hspace{4pt} \exp\left\{(1+o(1))\frac{s^kp}{(k-1)!}(1-p)^{s \choose k-1}\right\}  \\
    &\stackrel{\eqref{1_minus_p}}{\leq} \exp \left\{(1+o(1)) \frac{s^{2k}p}{(k-1)!e^kn^k}\right\} \\
    &\text{\hspace{4pt}}= \hspace{3pt}1+o(1),
\end{align*}
we have
\begin{align*}
    \mathbb{P}(Z_{T, m_T}&Z_{U, m_U} = 1) \leq (1+o(1)) \mathbb{P}(Z_{T, m_T} = 1) p^{q_O} (1-p)^{{\tilde{s} \choose k}-{i \choose k}-q_O} \left(j{\tilde{s} \choose k-2}p\right)^{2j-q_O} \\
&\stackrel{\text{L.}\ref{lm:bin_coeff_diff}}{\leq} (1+o(1)) \mathbb{P}(Z_{T, m_T} = 1) p^{q_O} (1-p)^{j(1-\varepsilon)\frac{f(\tilde{s})}{(k-1)!}-q_O} \left(j{\tilde{s} \choose k-2}p\right)^{2j-q_O}.
\end{align*}
Multiplying the probability estimate by the estimate for the number of choices for $m_T$ and $m_U$ given by~\eqref{m_choices} and summing over $q_O$ we see that for a fixed $T$ and $U$ we have
$$\frac{1}{\mathbb{P}(Z_{T, m_T}=1)}\sum\limits_{m_T, m_U} \mathbb{P}(Z_{T, m_T}Z_{U, m_U} = 1)$$
is at most
\begin{align*}
    &(1+o(1))\sum\limits_{q_O=0}^{\min\{r, j\}} \frac{4^j\tilde{s}^{kr+(k-2)q_O} (k!r)^{q_O}}{(k!)^r r!} p^{q_O} (1-p)^{j(1-\varepsilon)\frac{f(\tilde{s})}{(k-1)!}-q_O} \left(j{\tilde{s} \choose k-2}p\right)^{2j-q_O} \\
    &= (1+o(1)) \frac{\tilde{s}^{kr}}{(k!)^r r!} \left(4p^2j^2 {\tilde{s} \choose k-2}^2 (1-p)^{(1-\varepsilon)\frac{f(\tilde{s})}{(k-1)!}}\right)^j \sum\limits_{q_O=0}^{\min\{r, j\}} \left(\frac{\tilde{s}^{(k-2)}k!r}{j{\tilde{s} \choose k-2}(1-p)}\right)^{q_O} \\
    &\leq (1+o(1)) \frac{\tilde{s}^{kr}}{(k!)^r r!} \left(4p^2j^2 {\tilde{s} \choose k-2}^2 (1-p)^{(1-\varepsilon)\frac{f(\tilde{s})}{(k-1)!}}\right)^j \max\{1, (3k!(k-2)!r)^j\} \\
    &\leq (1+o(1)) \frac{\tilde{s}^{kr}}{(k!)^r r!} \left(12k!(k-2)!p^2j^2 {\tilde{s} \choose k-2}^2 (1-p)^{(1-\varepsilon)\frac{f(\tilde{s})}{(k-1)!}} \max\{1, r\}\right)^j \\
    &\leq \frac{(1+o(1))\tilde{s}!}{(\tilde{s}-kr)!(k!)^r r!} \left(12k!(k-2)!p^2j^2 {\tilde{s} \choose k-2}^2 (1-p)^{(1-\varepsilon)\frac{f(\tilde{s})}{(k-1)!}} \max\{1, r\}\right)^j,
\end{align*}
where we use $r=o(s^{1/k})$ by~\eqref{r_order} in the last step and assume $r^{q_O}=1$ if $r=q_O=0$. Recalling that $|\mathbb{T}| = {n \choose \tilde{s}} \frac{\tilde{s}!}{(\tilde{s}-kr)! (k!)^rr!}$ and that the number of choices $T, U$ is ${n \choose \tilde{s}}{n - \tilde{s} \choose j} {\tilde{s} \choose j}$, we get
\begin{align*}
    &h_i\mathbb{E} (Z) \leq \frac{1}{|\mathbb{T}| \mathbb{P}(Z_{T, m_T}=1)} \sum\limits_{|T \cap U|=i} \sum\limits_{m_T, m_U} \mathbb{P}(Z_{T, m_T}Z_{U, m_U}=1) \\
    &\text{\hspace{4pt}}\leq (1+o(1)) {n-\tilde{s} \choose j} {\tilde{s} \choose j} \left(12k!(k-2)!p^2j^2 {\tilde{s} \choose k-2}^2 (1-p)^{(1-\varepsilon)\frac{f(\tilde{s})}{(k-1)!}} \max\{1, r\}\right)^j \\ 
    &\text{\hspace{4pt}}\leq (1+o(1)) \left(\frac{n\tilde{s}e^2}{j^2}\right)^j \left(12k(k-1)p^2j^2 \tilde{s}^{2k-4} (1-p)^{(1-\varepsilon)\frac{f(\tilde{s})}{(k-1)!}} \max\{1, r\}\right)^j \\
    &\text{\hspace{4pt}}\leq (1+o(1)) \left(12k(k-1)e^2 n \tilde{s}^{2k-3} p^2 (1-p)^{(1-\varepsilon)\frac{f(\tilde{s})}{(k-1)!}} \max\{1, r\}\right)^j \\
    &\stackrel{\eqref{1_minus_p}}{\leq} (1+o(1)) \left(13k(k-1)e^{-(k-2)+k\varepsilon} n^{-(k-1)+k\varepsilon} \tilde{s}^{3k-3-k\varepsilon} p^2 \max\{1, r\}\right)^j.
\end{align*}
As $r=o(s^{1/k})$ and $n^{-(k-1)+k\varepsilon}s^{3k-3-k\varepsilon + 1/k}p^2 = o(1)$, we finally get $$\sum\limits_{i=f^{-1}((1-\varepsilon)f(s))}^{\tilde{s}-1} h_i = o(1).$$

\section{Two-point concentration for $n^{-(k-1)^2/k+\varepsilon}\ < p = o(1)$}
\label{sc: simple_sec}

In this section, we show that to prove two-point concentration of $\alpha(H(n, k, p))$ for sufficiently large $p$, it is sufficient to apply the second moment method to the random variable $Y_s$, which counts the number of \textit{maximal} independent sets of size $s$ in $H(n, k, p)$.

Recall that $X_s$ is the random variable which counts the number of ordinary independent sets of size $s$ in $H(n, k, p)$. However, in this section we define $s_x$ as the largest integer such that $\mathbb{E} (X_s) > \varepsilon \log(n)$. We have reduced the cut-off function in order that $s_x$ could still be represented in the form~\eqref{main_fm} (as can be easily verified), and $\mathbb{E}[X_{s_x}]=\omega(1)$. Also, for convenience, let $X=X_{s_x}$ and $Y=Y_{s_x}$.

\vspace{\baselineskip}

To work with $p$ close to a constant, we need to adapt the asymptotic expressions from Section~\ref{sc:tech} into inequalities. It is easy to verify that for larger $p$, expressions~\eqref{1_minus_p} and \eqref{pre_exp} will become
\begin{equation}
\label{1_minus_p_ineq}
    (1-p)^{\frac{f(s)}{k!}} \leq (1+o(1)) \frac{s}{ne};
\end{equation}
\begin{equation}
\label{pre_exp_ineq}
    \left(1-(1-p)^{s \choose k-1}\right)^{n-s} \geq e^{-(1+o(1))\frac{s^k}{e^kn^{k-1}}}.
\end{equation}

However, we will also be interested in inequalities in the opposite direction:
\begin{align*}
    \frac{ne}{s}(1-p)^{\frac{f(s)}{k!}} &\geq (1+o(1))\frac{ne}{s} \exp[-\log d + \frac{1}{k-1} \log \log d - \log\frac{e}{(k!)^{1/(k-1)}} - \\
    &\quad\quad\quad\quad\quad\quad\quad\quad\hspace{7pt}-p\log d+\frac{p}{k-1} \log \log d + o(1)]\\
    &=(1+o(1))\left(\frac{\log (d)^{1/(k-1)}}{d}\right)^p \geq (1+o(1)) \frac{1}{n^p}.
\end{align*}
From which it follows that
\begin{equation}
\label{1_minus_p_ineq2}
    (1-p)^{\frac{f(s)}{k!}} \geq (1+o(1)) \frac{s}{en^{1+p}};
\end{equation}

\vspace{\baselineskip}

Now, let us proceed directly to the main theorem of this section:
\begin{theorem}
\label{th:simple_sec}
    Consider $p=p(n)$ such that $n^{-\frac{(k-1)^2}{k} + \varepsilon} < p < o(1)$ for some fixed $\varepsilon > 0$. Then $\alpha(H(n, k, p)) \in \{s_x, s_x+1\}$ a.a.s..
\end{theorem}

\begin{proof}
    Without loss of generality, we assume that $\varepsilon < \frac{1}{k}$. Also, for convenience, set $s=s_x$.
    
    First, we show by the standard first moment method that no independent set of size $s+2$ appears whp:
    \begin{align*}
        \mathbb{P}(\alpha(H(n, k, p)) \geq s+2) &\hspace{4pt}\leq {n \choose s+2}(1-p)^{s+2 \choose k} \\
        &\hspace{4pt}= {n \choose s+1}(1-p)^{s+1 \choose k} \left(\frac{n-s-1}{s+2}(1-p)^{s+1 \choose k-1}\right) \\
        &\stackrel{\eqref{1_minus_p_ineq}}{\leq} (1+o(1)) {n \choose s+1} (1-p)^{s+1 \choose k} \frac{n}{s} \left(\frac{s}{ne}\right)^k  \\&\hspace{5pt}{=}\hspace{5pt} o\left(\frac{s^{k-1}}{n^{(k-1)^2/k}}\right) = o(1),
    \end{align*}
where in the end we use the definition of $s_x$ and the fact that $s^{k-1} = o\left(n^{(k-1)^2/k}\right)$ when $p > n^{-\frac{(k-1)^2}{k} + \varepsilon}$.

    Next, we show that an independent set of size $s$ exists whp, using the second moment method. But the variance of $X$ we are interested in is too large, that is why we work with the random variable $Y$ which counts the number of maximal independent sets of size $s$. Let us estimate the first moment of $Y$:
    \begin{equation}
    \label{y_exp}
        \begin{split}
        \mathbb{E} (Y) &\hspace{4pt}=\hspace{3pt} {n \choose s} (1-p)^{s \choose k} \left(1-(1-p)^{s \choose k-1}\right)^{n-s} \\
        &\stackrel{\eqref{pre_exp_ineq}}{\geq} {n \choose s} (1-p)^{s \choose k} e^{-O(s^k / n^{k-1})} = \mathbb{E} (X) (1-o(1)) = \omega(1).
    \end{split}
    \end{equation}

    By Chebyshev's Inequality, all that is left to show is that $\frac{{\sf Var}(Y)}{\mathbb{E}(Y)^2} = o(1).$ We write $Y$ as the sum of indicator variables $Y_T$ over all sets $T$ of size $s$, where $Y_T$ is the indicator for the event that $T$ is a maximal independent vertex set. We define indicator variables $X_T$ for ordinary independent sets analogously. Then we have
    $${\sf Var}(Y)=\sum\limits_{T} {\sf{Var}}(Y_T) + \sum\limits_{T \neq U} {\sf Cov}(Y_T, Y_U) \leq \mathbb{E} (Y) + \sum\limits_{T \neq U} \left[\mathbb{E} (Y_T Y_U) - \mathbb{E}(Y_T)\mathbb{E}(Y_U)\right].$$
    Next, let $\mathbb{T}$ be the collection of all $s$-element vertex subsets of $H(n, k, p)$ and define, for all relevant $i$:
    \begin{align*}
        f_i &:= \frac{1}{\mathbb{E} (X)^2} \sum\limits_{|T \cap U|=i} \mathbb{E} (X_T X_U) = \frac{{s \choose i}{n-s \choose s-i}}{{n \choose s}} (1-p)^{-{i \choose k}},\\
        g_i &:= \frac{1}{\mathbb{E}(Y)^2} \sum\limits_{|T \cap U|=i} \mathbb{E} (Y_T Y_U), \text{ and}\\
    h_i &:= \frac{1}{\mathbb{E}(Y)^2} \sum\limits_{|T \cap U|=i} \mathbb{E} (Y_T Y_U)  - \mathbb{E}(Y_T) \mathbb{E}(Y_U) = \frac{1}{|\mathbb{T}|^2} \sum\limits_{|T \cap U|=i} \left(\frac{\mathbb{E} (Y_U | Y_T=1)}{\mathbb{E}(Y_T)} - 1\right).
    \end{align*}

Since $\mathbb{E}(Y)=\omega(1)$, our goal is to show that
$$\sum\limits_{i=1}^{k-1} h_i = o(1).$$
We consider four ranges of $i$ for each of which we show that the sum of variables $h_i$ within that range is $o(1)$. For the second and fourth ranges, it will be sufficient for us to show this for the sum of incomplete covariances $g_i$ since it is obvious that $h_i \leq g_i$. For the third range, we will work with the sum $f_i$, noting that $Y_T \leq X_T$, and therefore, by the same argument as in inequality~\eqref{y_exp}, we have $h_i \leq g_i \leq f_i(1+o(1))$. Note that conceptually the proof for all cases fully coincides with the proof in Section~\ref{sc:second_moment}.
\\
Before analyzing the cases, let us estimate $g_i$ to obtain an expression analogous to the right-hand side of the formula for $f_i$ above.
 Consider two arbitrary sets of vertices $T$ and $U$ intersecting in $i$ vertices. Given that they are both independent sets, the probability that both sets are maximal independent is bounded above by the probability that for any vertex $x \in T\backslash U$, there is an edge $e \subset x \cup U$, $e \nsubseteq T$ that contain $x$.
 This probability is $\left(1-(1-p)^{{s \choose k-1} - {i \choose k-1}}\right)^{s-i} \leq \left(\left({s \choose k-1} - {i \choose k-1}\right)p\right)^{s-i} \leq \left((s-i){s-1 \choose k-2}p\right)^{s-i}$ (while it is, of course, bounded by 1). Therefore, we have
\begin{align*}
    g_i &\leq (1+o(1))\frac{{s \choose i}{n-s \choose s-i}}{{n \choose s}} (1-p)^{-{i \choose k}}  \min\left[\left((s-i){s-1 \choose k-2}p\right)^{s-i}, 1\right].
\end{align*}

So, let us start wit the first range $i < B:=(p\log (n))^{-1/k}$. Note that in this range we have $(1-p)^{-{i \choose k}} = 1+o(1)$. Therefore, $g_i \leq (1+o(1)) \frac{{s \choose i} {n-s \choose s-i}}{{n \choose s}}$ and we get
\begin{align*}
    h_i \leq \frac{o(1) {s \choose i} {n-s \choose s-i}}{{n \choose s}},
\end{align*}
and so $\sum\limits_{i=0}^B h_i = o(1)$.

Next, we consider $B < i \leq B^{k/(k-1)}$ (note that for large $p > 1/\log(n)$ this range can represent an empty set). We have:
$$g_i \leq \frac{{s \choose i} {n-s \choose s-i}}{{n \choose s}} \frac{1+o(1)}{(1-p)^{i \choose k}} \leq (1+o(1))\left(\frac{es^2}{ni(1-p)^{f(i)/k!}}\right)^i = (1+o(1)) \left(\frac{e^2s^2 p^{1/k} \log^{1/k} (n)}{n}\right)^i.$$
Since $\frac{s^2p^{1/k}\log^{1/k} (n)}{n} = o(1)$, then the sum of $g_i$ within the second range is $o(1)$.

Next, consider the range $(p \log n)^{-1/(k-1)} \leq i \leq f^{-1}((1-\varepsilon)f(s))$:
\begin{align*}
    f_i &\hspace{4pt}\leq\hspace{4pt} \left(\frac{es}{i}\right)^i \left(\frac{s}{n}\right)^i (1-p)^{-{i \choose k}}  \\
    &\hspace{4pt}\leq \hspace{4pt}\left[\frac{es^2}{ni} \left((1-p)^{-f(s)/k!}\right)^{f(i-1)/f(s)}\right]^i  \\
    &\stackrel{\eqref{1_minus_p_ineq2}}{\leq} \left[(1+o(1)) \frac{es^2 p^{1/(k-1)} \log^{1/(k-1)} (n)}{n} \left(\frac{en^{1+p}}{s}\right)^{1-\varepsilon}\right]^i \\
    &\hspace{4pt}=\hspace{4pt} \left[(1+o(1)) e^{2-\varepsilon} \frac{s^{1+\varepsilon}p^{1/(k-1)} \log^{1/(k-1)} (n)}{n^{\varepsilon}} n^{p(1-\varepsilon)}\right]^i.
\end{align*}
As $p=o(1)$ then $s^{1+\varepsilon}p^{1/(k-1)} \log^{1/(k-1)} (n) n^{p(1-\varepsilon)} = o(n^{\varepsilon})$ and consequently, the sum of $f_i$ over this range is again $o(1)$.

Finally, consider the range $i > f^{-1}((1-\varepsilon)f(s))$. We have
\begin{align*}
    g_i &\hspace{5pt}\leq\hspace{4pt} (1+o(1))\frac{{s \choose i}{n-s \choose s-i}}{{n \choose s}} (1-p)^{-{i \choose k}} \left((s-i){s-1 \choose k-2}p\right)^{s-i} \\
    &\hspace{5pt}\leq\hspace{4pt} \frac{(1+o(1))}{\mathbb{E}[X]} \left({s \choose s-i} \left((s-i)\frac{f(s)}{s(k-2)!}\right)^{s-i}\right) {n-s \choose s-i} (1-p)^{{s \choose k}-{i \choose k}} p^{s-i} \\
    &\stackrel{\text{L.} \ref{lm:bin_coeff_diff}}{\leq} \frac{(1+o(1))}{\mathbb{E}[X]} \left(\frac{e f(s)}{(k-2)!}\right)^{s-i} n^{s-i} (1-p)^{(s-i)(1-\varepsilon)\frac{f(s)}{(k-1)!}} p^{s-i} \\
    &\hspace{2pt}\stackrel{\eqref{1_minus_p_ineq}}{\leq} \frac{1}{\mathbb{E}[X]} \left((1+o(1))\frac{ef(s) np}{(k-2)!} \left(\frac{s}{ne}\right)^{k(1-\varepsilon)}\right)^{s-i}  \\
    &\hspace{5pt}\leq \hspace{4pt} \frac{1}{\mathbb{E}[X]} \left(k^2\log^3(n)n^{-\frac{1}{k-1}\varepsilon+\frac{k}{k-1}\varepsilon^2}\right)^{s-i}.
\end{align*}
Hence, for a sufficiently small constant $\varepsilon$ (we assumed it at the beginning of the proof), the sum $g_i$ over the last range is also $o(1)$, which completes the proof.

\end{proof}
\section{Absence of two-point concentration for $n^{-(k-1)} < p < n^{-\frac{(k-1)k}{k+1}}$}
\label{sc:anti}

In this section we prove Theorem~\ref{th:antic}. It follows from a more general theorem on the absence of two-point concentration for all \textit{$j$-independence} numbers:
\begin{theorem}
\label{tm:anti}
   Fix $\varepsilon > 0$, $k \geq 2$ and $1 \leq j \leq k-1$. Let  $p = p(n)$ satisfy the conditions $p=\omega\left(n^{-(k-1)}\right)$ and \\$p < \left(\left(1-\frac{1}{2^{1/j}}-\varepsilon\right) n^{1-\frac{k}{2} - \frac{k-1}{j}}/2\right)^{2j/(j+2)} \left(\frac{(j+1)(j-1)!(k-1)(k-1-j)! \log (n)}{k+1}\right)^{2/(j+2)}$. \\Set $$\ell = \left(1-\frac{1}{2^{1/j}}-\varepsilon\right) n^{1-\frac{k}{2} - \frac{k-1}{j}}p^{-1/2-1/j} \left((j+1)(j-1)!(k-1)(k-1-j)! \log (d)\right)^{1/j}.$$ Then there exists $q = q(n)$ lying between $p$ and $2p$ such that $\alpha_j(H(n, k, q))$ is not concentrated on $\ell$ values.

\end{theorem}
\begin{remark}
The upper bound on $p$ is chosen such that the corresponding $\ell \geq 2$. Thus, in this range two-point concentration does not hold in general. Moreover, for $p = O\left(n^{-k+\frac{2j+2}{j+2}}\right)$, the order $\ell$ is already equal to $\Omega(\log^{1/j} (d(n, p)))$.
\end{remark}
\begin{proof}
    
For $p$ in the specified range we define
$$p' = p+n^{-k/2}\sqrt{p}.$$
We first observe that this choice of $p'$
is close enough to $p$ to ensure that there is no ``separation''
of the intervals over which $\alpha_j(H(n, k, p))$ and $\alpha_j(H(n, k, p'))$ are respectively concentrated:
\begin{lemma}
\label{lm:segm_division}
    If $\omega\left(n^{-\frac{1}{k-1}}\right) < p < o(1)$ then there is a sequence $a=a(n)$ such that
    $$\mathbb{P} [\alpha_j(H(n, k, p)) \leq a] > \frac{1}{20} \quad \text{and} \quad \mathbb{P} [\alpha_j(H(n, k, p')) \geq a] > \frac{1}{20}$$
   for $n$ sufficiently large.
\end{lemma}
\begin{proof}
As well known, the number of edges in a random hypergraph $e(H(n, k, p))$ has the binomial distribution ${\sf Bin}({n \choose k}, p)$. The key observation is that for large $n$ the distributions of $e(H(n, k, p))$ and $e(H(n, k, p'))$ are approximately Gaussian with equal variances and with means that are $\frac{1}{\sqrt{k!}}$ standard deviation apart.
 Therefore, there is some value $m$ such that
    $$\mathbb{P} [e(H(n, k, p)) > m] > \frac{1}{10} \quad \text{и} \quad \mathbb{P} [e(H(n, k, p')) < m] > \frac{1}{10}.$$
    Then we take $a$ as the median value of $j$-independence number of the random hypergraph in the uniform model with $m$ edges. Since adding edges to a hypergraph does not increase independence numbers, the lemma is proved.
    
\end{proof}

Now, let us proceed to the proof of Theorem~\ref{tm:anti}.
 Assume the contrary: for sufficiently large $n$ and  $p \leq q \leq 2p$ there exists an interval $I$ of length at most $\ell$ such that

$$\mathbb{P} [\alpha_j(H(n, k, p)) \in I] > \frac{49}{50}.$$
Then consider the sequence $p_0, p_1, p_2, \ldots$, where $p_0 = p$ and $p_{i+1} = p_i'$ for $i \geq 0$. Let $z$ be the lowest index such that $p_z \geq 2p$; it is easy to see that $z \leq n^{k/2}\sqrt{p}$. Let $I_0 = [a_0, b_0], I_1 = [a_1, b_1], \ldots, I_{z-1} = [a_{z-1}, b_{z-1}]$ be intervals of $\ell$ integer values such that $\alpha_j(H(n, k, p_i))$ lies in $I_i$ with probability at least $\frac{49}{50}$ for each $i$. By Lemma~\ref{lm:segm_division} we have $a_{i} \leq b_{i+1}$ for all $i < z$. This implies $b_{i+1} \geq b_i - \ell$. Iterating this observation gives
\begin{align*}
    b_z &\geq b_0 - z\ell \\
    &\geq b_0 - \left(1-\frac{1}{2^{1/j}}-\varepsilon\right) n^{1-\frac{k-1}{j}} \left((j+1)(j-1)!(k-1)(k-1-j)!\frac{\log (np^{1/(k-1)})}{p}\right)^{1/j}.
\end{align*}

On the other hand, using the asymptotic relation for the independence number in~\eqref{asymptotic_base}, we have
\begin{align*}
    b_0 - b_z &= (1+o(1)) n^{1-\frac{k-1}{j}} \left((j+1)(j-1)!(k-1)(k-1-j)!\frac{\log (np^{1/(k-1)})}{p}\right)^{1/j} - \\
    & -(1+o(1)) n^{1-\frac{k-1}{j}} \left((j+1)(j-1)!(k-1)(k-1-j)!\frac{\log \left(n(2p)^{1/(k-1)}\right)}{2p}\right)^{1/j} \\
    &= \left(1-\frac{1}{2^{1/j}}+o(1)\right) n^{1-\frac{k-1}{j}} \left((j+1)(j-1)!(k-1)(k-1-j)!\frac{\log (np^{1/(k-1)})}{p}\right)^{1/j}.
\end{align*}
$$$$

This is a contradiction with the upper bound above for $b_0-b_z$.
\end{proof}

\section{Conclusion}
\label{sc:further}

It is clear from the proof of Theorem~\ref{th:main_res} that the lower bound on $p$ can be relaxed to $n^{-\frac{(k-1)k}{(k+1)}} \log^{\beta} (n)$ for some $\beta > 0$ which would be closer to the lower bound obtained in Theorem~\ref{th:antic}. To achieve this, one would only need to change the value of the ``cut-off functions'' in the definition of $s_x$ and $s_z$. This was not done only due to the increased number of technical calculations.

In the future, it would be nice to prove a similar result concerning the concentration of $j$-independence numbers for $1 \leq j \leq k-2$, using the same augmentation set technology. In this case, however, it will be necessary to consider edges containing vertices outside the potential independent sets which significantly complicates the calculation of the second moment. We believe that the following result holds true:
\begin{hypoth}
    Fix $k\geq 3$ and $1 \leq j \leq k-2$. Let $p=p(n) \geq n^{-k+\frac{2(j+1)}{(j+2)}+\varepsilon}$ for some $\varepsilon > 0$. Then $\alpha_{j}(H(n, k, p))$ is concentrated on two values.
\end{hypoth}

It would also be interesting to obtain our result on the asymptotic behavior of the independence number --- Theorem~\ref{th:asympt} --- using the technology of A. Frieze in~\cite{frieze_asympt}, who applied Azuma's inequality to random graphs in a surprising way.

\section{Acknowledgments}
The author thanks Dmitry A. Shabanov for posing the problem and the anonymous referees for their helpful comments and suggestions, which improved the clarity and quality of the paper.

\end{document}